\documentclass{article}
\usepackage[utf8]{inputenc}

\usepackage{amsmath}
\usepackage{amsfonts}
\usepackage{amsthm}
\usepackage{verbatim}
\usepackage{subfig}
\usepackage{graphicx}
\usepackage{ulem}
\usepackage{hyperref}
\usepackage{cleveref}

\usepackage{amssymb}

\def\R{{ \mathbb{R}}}

\newcommand{\ep}{\epsilon}

\newtheorem{thm}{Theorem}[section]
\newtheorem{cor}{Corollary}[section]

\newtheorem{theorem}{Theorem}
\newtheorem{lemma}[cor]{Lemma}
\newtheorem{proposition}[cor]{Proposition}

\newtheorem{corollary}[cor]{Corollary}

\theoremstyle{definition}
\newtheorem{definition}[cor]{Definition}

\usepackage{color}

\newtheorem{remark}[cor]{Remark}

\newtheorem*{remark*}{Remark}
\usepackage{epstopdf}
\usepackage{float}
\numberwithin{figure}{section} 
\numberwithin{equation}{section}

\begin{document}

\title{Asymptotic spreading of KPP reactive fronts in heterogeneous shifting environments
}
\author{King-Yeung Lam \\
{\small Department of Mathematics, The Ohio State University, Columbus, OH 43210, USA}\\
\\
Xiao Yu\footnote{Corresponding Author: xymath19@m.scnu.edu.cn}\\
{\small School of Mathematical Sciences, South China Normal University, Guangzhou 510631,  China}}

\date{July 2020}
\date{}
\maketitle

\begin{abstract}
    We study the asymptotic spreading of Kolmogorov-Petrovsky-Piskunov (KPP) fronts in heterogeneous shifting habitats, with any number of shifting speeds, by further developing the method based on the theory of viscosity solutions of Hamilton-Jacobi equations. Our framework addresses both reaction-diffusion equation and integro-differential equations with a distributed time-delay. The latter leads to a class of limiting equations of Hamilton-Jacobi-type depending on the variable $x/t$ and in which the time and space derivatives are coupled together. We will first establish uniqueness results for these Hamilton-Jacobi equations using elementary arguments, and then characterize the spreading speed in terms of a reduced equation on a one-dimensional domain in the variable $s=x/t$. In terms of the standard Fisher-KPP equation, our results leads to a new class of ``asymptotically homogeneous" environments which share the same spreading speed with the corresponding homogeneous environments.

    

\end{abstract}
%
\section{Introduction}

In this paper, we consider the spreading property of positive solutions $u(t,x)$ to a class of  functional differential equations with diffusion on  $\mathbb{R}$:
\begin{equation}\label{Model}
\begin{cases}
\partial_t u=\! \partial_{xx} u\!+\! f_1(t,x,u)\!+\! \int_{0}^{\tau_0}\!\!\int_{\mathbb{R}}\Gamma(\tau,y)f_2(t\!-\!\tau,x\!-\!y,u(t\!-\!\tau,x\!-\!y)) dy d\tau,  & \text{ in }(0,\infty)\! \times\! \mathbb{R},\\
u(t,x) =\! \phi(t,x), &\text{ on } [-\tau_0,0]\! \times \!\mathbb{R},
\end{cases}
\end{equation}
where $0<\tau_0<\infty$. The existence and uniqueness of solutions can be established as in \cite{XuXiao_2016}. When $f_1(t,x,u)=(r(t,x)-u)u$ and $f_2\equiv 0$, we give new results concerning the classical Fisher-KPP equation with heterogeneous coefficients. When $f_2$ is non-trivival and $f_1<0$, the model \eqref{Model} was motivated by the study of structured populations with distributed maturation delay, in which juveniles and adults have different movement patterns and $f_2$,$-f_1$ are regarded as birth and death functions of adult population, respectively; see \cite{Gourley2003,GW2006,Schaaf_1987,SWZ_2001,XuXiao_2016} and the references therein.

We will treat the following classes of initial data $\phi$.
\begin{definition}\label{def:initial}
We say that the initial data $\phi$ satisfies (IC$_\mu$) for some $\mu \in (0,\infty)$ provided $\phi \!\in\! L^\infty([-\tau_0,0]\times\mathbb{R})$ is non-negative and there is $0<a<b$ such that
$$
a  e^{-(\mu+o(1))x} \le \phi(t,x) \leq b e^{-(\mu + o(1))x}  \quad \text{ for }x  \gg 1, \forall t\in[-\tau_0,0].
$$
We say that the initial data $\phi$ satisfies (IC$_\infty$) provided $\phi \in L^\infty([-\tau_0,0]\times\mathbb{R})$ is non-negative, 
and 
$$
\limsup_{ x \to \infty} \left[e^{\mu x}\sup\limits_{t \in [-\tau_0,0]}\phi(t,x)\right] = 0 \quad \text{ for every }\mu>0.
$$
In particular, $\phi$  satisfies $({IC}_\mu)$ for some $\mu \in (0,\infty)$ if there are positive constants $a$ and $b$ such that $a  e^{-\mu x} \le \phi(t,x) \leq b e^{-\mu x}$; whereas $\phi$ satisfies $({IC}_\infty)$
 if it is compactly supported in $[-\tau_0,0]\times\mathbb{R}$.
\end{definition}
The estimation of the asymptotic speeds of spread, or spreading speeds, is central in the study of biological invasions. 
The concept, originally introduced by Aronson and Weinberger \cite{Weinberger1978}, says that a species residing in a one-dimensional domain $\mathbb{R}$ with population density $u(t,x)$ has spreading speed $c^*>0$ if, for each $\eta>0$,
$$
\lim_{t\to\infty} \left[\sup_{x \geq (c^*+\eta)t} u(t,x)\right] = 0\quad \text{ and }\quad \liminf_{t\to\infty} \left[\inf_{0 \leq x \leq (c^*-\eta)t} u(t,x) \right]>0.
$$
For the Fisher-KPP equation, 
\begin{equation}\label{eq:model4}
\partial_t u= \partial_{xx} u +u(r(t,x)-u)  \quad  \text{ in }(0,\infty) \times \mathbb{R},
\end{equation}
it is well-known \cite{Fisher1937,Kolmogorov_1937,Weinberger1978} that when $r(t,x) \equiv r_0$ for some positive constant $r_0$, then the single species has spreading speed $c^* = 2\sqrt{r_0}$.
In addition, the spreading speed $c^*$ coincides with the minimal speed of the
traveling wave solutions to \eqref{eq:model4}.

For spatially periodic environment, Weinberger \cite{Weinberger_1982} introduced an elaborate method and showed the existence of spreading speeds for discrete-time recursions. Subsequently, the general theory on the existence of spreading speeds and its coincident with the minimal wave speed was developed in \cite{LiangZhao2007} for monotone dynamical systems and \cite{FYZ_2017} for the time-space periodic monotone systems with weak compactness. 

Using the super- and sub-solutions method and the principal eigenvalue of time-space periodic parabolic problems, spreading properties in time-space periodic and more general environments are investigated in \cite{BHN2008}, as well as in \cite{Shen_2010}. 



More recently, by combining the Hamilton-Jacobi approach \cite{ES89} and homogenization ideas, Berestycki and Nadin \cite{Nadin2012,Nadin2020} showed the existence of spreading speeds for spatially almost periodic, random stationary ergodic, and very general environments. Their spreading speed is expressed as a minimax formula in terms of suitable notions of generalized principal eigenvalues in unbounded domains. See also \cite{Zhou2020} for the relative result on the nonlocal KPP models. 

In the following, we outline the basic ideas of the Hamilton-Jacobi approach.

\subsection{The Hamilton-Jacobi approach}

\medskip

The Hamilton-Jacobi approach was introduced by Freidlin \cite{Freidlin1984}, who employed probabilistic
arguments to study the asymptotic behavior of solution to the Fisher-KPP equation
modeling the population of a single species. Subsequently, the result was generalized
by Evans and Souganidis \cite{ES89} using PDE arguments; see also 
\cite{Xin2000}. 
The method can be briefly outlined as follows:
\begin{enumerate}
\setlength\itemsep{0em}
    \item The WKB-Ansatz: $u^\ep(t,x) = u\left(\frac{t}{\ep},\frac{x}{\ep}\right)$ and $w^\ep(t,x) = - \ep \log u^\ep(t,x)$.
    \item Observe that the half-relaxed limits $w^*(t,x)$ and $w_*(t,x)$, given by
\begin{equation}\label{eq:wstar'}
w^*(t,x)=\limsup_{\epsilon\to0
	\atop(t',x')\to(t,x)} w^\epsilon(t',x'),\quad w_*(t,x)=\liminf_{\epsilon\to0
	\atop(t',x')\to(t,x)} w^\epsilon(t',x'),
\end{equation}
    are respectively viscosity sub- and super-solutions of a given Hamilton-Jacobi equation \cite{BP1987}.
    \item Show that $w^*(t,x) \leq  w_*(t,x)$ by establishing a strong comparison result (SCR). 
    \item Since $w^* \geq w_*$ by construction, we have $w^*=w_*$ and thus $w^\epsilon$ converges locally uniformly to the unique vicosity solution $\hat{w}(t,x)$ of the Hamilton-Jacobi equation.
    \item Since $\hat{w}(t,x) \geq 0$ for all $(t,x)$, 
    the spreading speed $c^*$ can be characterized by the free boundary separating the regions $\Omega_0:=\textup{Int}\,\{(t,x):\, \hat{w}(t,x)=0\}$ and $\Omega_1:= \{(t,x):\, \hat{w}(t,x)>0\}$.
\end{enumerate}

For example, when $r(t,x) \equiv r_0$ in the Fisher-KPP equation \eqref{eq:model4} with Heaviside-like initial data, the limiting Hamilton-Jacobi equation is
\begin{equation}\label{eq:hj_kpp_0}
\min\{w, w_t + |w_x|^2 + r_0\} = 0 \quad \text{ for }(t,x) \in (0,\infty)\times \mathbb{R},
\end{equation}
with initial condition
\begin{equation}\label{eq:initialinf}
w(0,x) = 0 \quad \text{ for }x \leq 0, \quad\text{ and }\quad w(0,x) = +\infty \quad \text{ for }x > 0.
\end{equation}
By a duality correspondence of the viscosity solution of \eqref{eq:hj_kpp_0} with the value function of certain zero sum, two player differential game with stopping times, and the dynamic programming principle, it can be shown \cite{ES89} that \eqref{eq:hj_kpp_0} with initial data \eqref{eq:initialinf} has a unique viscosity solution 
$$
\hat{w}(t,x) = t\max\left\{ \frac{1}{4} \left| \frac{x}{t}\right|^2 - r_0  , 0 \right\}.
$$
Thus $c^* = 2\sqrt{r_0}$ in this case. (Note that $L(v) = \frac{v^2}{4} - r_0$ is the Legendre transform of $H(p) = p^2 + r_0$.)
The above formula for spreading speed holds also for environments which are compact perturbations of homogeneous environment, i.e. $r(t,x) = r_0$ for $(t,x)$ outside a compact set \cite{BHN2008,KS2011}. 
\begin{remark}
{\it As we shall see, the above holds for a more general class of ``asympoptotically homogeneous" environments, namely, $r(t,x)$ satisfying $\liminf\limits_{M \to \infty} \inf\limits_{[M,\infty)\times [M,\infty)} r >0$, and 
$$
\limsup_{\epsilon\to0\atop(t',x')\to(t,x)} r\left(\frac{t'}{\ep},\frac{x'}{\ep} \right) = r_0 \quad \text{everywhere, and}\,
\liminf_{\epsilon\to0
	\atop(t',x')\to(t,x)} r\left(\frac{t'}{\ep},\frac{x'}{\ep} \right) = r_0\, \text{ almost everywhere}.$$
	See Theorem \ref{thm:kpp1}.  An example is $r(t,x) = r_0 - r_1(x-c_1 t)$ for some positive constants $r_0, c_1$ and a compactly supported, non-negative function $r_1$; see Figure \eqref{fig}. }
\end{remark}

When $r(t,x) = r(x)$ and is $1$-periodic, it is shown in \cite{Nadin2012} that the limiting H-J equation is
$$
\min\{w, w_t + \hat H(w_x)\}=0  \quad \text{ for }(t,x) \in (0,\infty)\times \mathbb{R},
$$
where $\hat H(p) = k^{per}_p$ is characterized as the principal eigenvalue of the elliptic eigenvalue problem
$$
\phi''- 2p \phi' + [r(x) + p^2] \phi =  k\phi \quad \quad \text{ and }\quad  \phi(x) = \phi(x+1)\quad \text{ in }\mathbb{R}.
$$
Then $\hat w(t,x) = t \max\left\{ \hat{L} \left( \frac{x}{t} \right) ,0 \right\}$, where $\hat L(v)$ is the Legendre transform of $\hat H(p)$, given by 
$$
\hat L(v):= \sup_{p>0} \left\{ vp - \hat{H}(p)\right\}.
$$
Since $w(t,c^*x) = 0$, we need to solve $\hat L(c^*) = 0$, i.e. $\sup\limits_{p>0} \left\{ c^*p - \hat{H}(p)\right\} = 0.$ 
Hence, 
$$
 c^* = \inf\limits_{p>0} \frac{\hat{H}(p)}{p} = \inf\limits_{p>0} \frac{k^{per}_p}{p},
$$
which gives an alternative derivation of the results in \cite{BHN2008,Freidlin1984,Weinberger2002}. This framework is substantially generalized recently by Berestycki and Nadin \cite{Nadin2012,Nadin2020} to almost periodic, random stationary ergodic, and more general environments, via the homogenization point of view using suitable notions of principal eigenvalues in unbounded domains of the form $\{(t,x): t > M,\, x>M\}$. 

On the other hand, the application of the Hamilton-Jacobi framework has largely been limited to reaction-diffusion or nonlocal diffusion equation \cite{Bouin2018}, and has not been extended to models with time-delay such as \eqref{Model}. 

\subsection{Asymptotic spreading in shifting environments}

Yet another type of spatio-temporal heterogeneity is introduced by the recent work of Potapov and Lewis \cite{Potapov2004} and Berestycki et al. \cite{Berestycki2009} to model the effect of shifting of isotherms. Such heterogeneities, incorporating the variable $x-ct$ in the coefficients, are not considered in the aforementioned results. 
 By assuming that the moving source patch for a focal species is finite and is being surrounded by sink patches, \cite{Berestycki2009,Potapov2004} investigated the critical patch size for species persistence. 
In \cite{Li2014}, Li et al. proposed to study the Fisher-Kpp equation with a shifting habitat $r(t,x) = r(x-c_1 t)$: 
\begin{equation}\label{eq:model5}
u_t = u_{xx} + u(r(x- c_1t)-u)  \quad  \text{ in }(0,\infty) \times \mathbb{R},  
\end{equation}
which describes the situation that the favorable
environment is shrinking in the sense that $c_1 > 0$ and $r \in C(\mathbb{R})$ is increasing and satisfies $r(-\infty)<0<r(+\infty)$. It is proved in \cite{Li2014} (see also \cite{Hu2020}) that if the species persists, then the species spreads at the speed $2\sqrt{r(+\infty)}$. 
We refer to \cite{Fang2018,FLW_2016} for the existence of forced waves, and to 
\cite{Zhang2017,Yuan2019} for related results for two-competing species. 
We also mention
\cite{BR_2008} for habitats with two-shifts.

More recently, the general theory on the propagation dynamics without spatial translational invariance was established by Yi and Zhao \cite{Yi2020} for monotone evolution systems. A key assumption in \cite{Yi2020} is that the given monotone system is sandwiched by two limiting homogeneous systems in certain translation sense, and that one of the limiting homogeneous system is unsuitable for species persistence while the other one has KPP structure. 
It was shown that the spreading speed coincides with the spreading speed in the limiting homogeneous systems with KPP structure. In particular, \cite{Yi2020} generalizes \cite{Li2014}
the context of \eqref{eq:model5}. 


An interesting case 
arises when both of the limiting systems has KPP structures, but with different spreading speeds, e.g. $c_-=2\sqrt{r(-\infty)},c_+=2\sqrt{r(+\infty)}$ for \eqref{eq:model5}. 
The spreading behavior when $0 < r(-\infty) < r(+\infty)$ is especially subtle. In \cite{HSL2019}, it is proved that $c^*=c_+$ if $c_1\ll 1$ and $c^* = c_-$ if $c_1 \gg 1$. But the general case remains open.  
By the maximum principle, it is not difficult to see that the actual spreading speed of the species must be no slower than $c_-$ and no faster than $c_+$. However, there is a fundamental difference between the homogeneous and heterogeneous cases as far as the spreading speed is concerned. As discussed earlier, the spreading speed can be computed locally when the environment is homogeneous or periodic. However, when the environment is heterogeneous and shifting, it is not always possible to calculate the spreading speed using local considerations \cite{Majda1994}. 
By the Hamilton-Jacobi approach, we can gain a more "global" point of view and show that the spreading speed of \eqref{eq:model5} can be subject to nonlocal pulling effect \cite{Scheel_2014,Girardin2019}, and is influenced by the speed $c_1$ of the shifting environment:
$$
c^* = \begin{cases}
c_+=2\sqrt{r_2} &\text{ if }c_1 \leq 2\sqrt{r_2},\\
\frac{c_1}{2}-\sqrt{r_2-r_1} + \frac{r_1}{\frac{c_1}{2}-\sqrt{r_2-r_1}} &\text{ if } 2\sqrt{r_2} < c_1 < 2(\sqrt{r_2-r_1} + \sqrt{r_1}),\\
c_-=2\sqrt{r_1} &\text{ if }c_1 \geq 2(\sqrt{r_2-r_1} + \sqrt{r_1}),
\end{cases}
$$
where $r_2 = r(+\infty)$, $r_1 = r(-\infty)$ and $r_2 > r_1 >0$. 
See Theorem \ref{thm:kppp}(iv) for details. We point out that it is possible to derive
this particular result as a consequence of \cite{Scheel_2014}, which relies on the change of coordinates $x' = x-c_1t$ to transform \eqref{eq:model5} into a problem with spatially heterogeneous, but temporally constant coefficients. But what about environments with more than one speed of shift, such as $r(t,x)=r_1(x-c_1t)+r_2(x-c_2t)$?

\subsection{Main Results}

In this paper, we will further develop the method based on Hamilton-Jacobi equations to determine the spreading speed of a species in a heterogeneous shifting habitat, with multiple (or indeed infinitely many) shifting speeds, which leads to a new class of Hamilton-Jacobi equations.  
Our approach will provide a unified framework to address both reaction-diffusion equation, and integro-differential equation with exponentially decaying or compactly supported initial data; see Definition \ref{def:initial}.
The spreading speed will be characterized in terms of a reduced Hamilton-Jacobi equation in a single variable $s=x/t$. 
We will also provide a new proof of uniqueness for the underlying Hamilton-Jacobi equations with elementary arguments. As a by-product of our approach, we obtain a new class of ``asymptotically homogeneous" environments which share the same spreading speed with the corresponding homogeneous environment; see Theorem \ref{thm:kpp1}.




For $i=1,2$, let  $R_i, \underline{R}_i\in L^\infty(\mathbb{R})$ be given by (note that $R_i,\underline{R}_i$ only depend on $x/t$ below) 
\begin{equation}\label{eq:h4}
\begin{cases}
R_i(x/t):= \limsup\limits_{\epsilon\to0
\atop(t',x')\to(t,x)}\partial_u f_i(t'/\ep,x'/\ep,0) & \text{ for  $(t,x) \in (0,\infty)\times\R$ }\\
\underline{R}_i(x/t):=	\liminf\limits_{\epsilon\to0
		\atop(t',x')\to(t,x)}\partial_u f_i(t'/\ep,x'/\ep,0) & \text{ for  $(t,x) \in (0,\infty)\times\R$. }
\end{cases}
\end{equation}
We recall the concept of local monotonicity from \cite{CH_2008}.
\begin{definition}\label{def:lc}
We say that $h: \mathbb{R} \to \mathbb{R}$ is locally monotone if, for each $s_0$, either
$$
\lim_{\delta \to 0} \mathop{\inf\limits_{|s_i-s_0|<\delta}}_{s_1<s_2} (h(s_1)-h(s_2)) \geq 0, \quad \text{ or }\quad \lim_{\delta \to 0} \mathop{\sup\limits_{|s_i-s_0|<\delta}}_{s_1<s_2} (h(s_1)-h(s_2)) \leq 0.
$$
\end{definition}
We will assume the following concerning \eqref{Model}:
\begin{enumerate}
\item[(H1)] For some constant $\delta^*>0$, $f_1\in C(\mathbb{R}_+ \times \mathbb{R} \times \mathbb{R}_+) \cap C^1(\mathbb{R}_+ \times \mathbb{R}\times [0,\delta^*])$ and $f_2\in C([-\tau_0,\infty) \times \mathbb{R} \times \mathbb{R}_+) \cap C^1([-\tau_0,\infty) \times \mathbb{R}\times [0,\delta^*))$ satisfy
$$
\begin{cases}
f_i(t,x,0) = 0 \quad \text{ and }\quad  f_i(t,x,u) \leq u \partial_u f_i(t,x,0) &\text{ for all }(t,x,u),\, i=1,2,\\
\sup\limits_{\R_+\times\mathbb{R} \times [0,M]}|\partial_uf_i(t,x,s)|<\infty & \text{ for each $M>0$, }i=1,2,\\
\end{cases}
$$
Furthermore, for any $\eta'>0$, there exists  $\delta_*>0$ independent of $(t,x)$ such that
\begin{equation}\label{eq:H2}
f_i(t,x,u) \geq u(\partial_u f_i(t,x,0)-\eta'), \quad \text{ if }0 \leq u \leq \delta_*,\, i=1,2.
\end{equation}
	\item[(H2)] There exists  $L_0>0$ such that $G(t,x,L,L) \leq 0$ for $(t,x) \in (0,\infty)\times \mathbb{R}$ and $L \in [L_0,\infty)$, where 
	$$G(t,x,u,v):= f_1(t,x,u)+ \int_0^{\tau_0}\int_\mathbb{R}\Gamma(\tau,y)  f_2(t-\tau,x-y,v)\,dyd\tau.$$
	\item[(H3)] $f_2(t,x,v) \equiv 0$, or $f_2(t,x,v) >0$ in $[-\tau_0,\infty)\times\mathbb{R}\times(0,\infty)$ and $\Gamma\in L^1([0,\tau_0]\times\R)$ is non-negative and satisfies $\int_0^{\tau_0} \int_{\mathbb{R}} \Gamma(\tau,y) \,dyd\tau=1$ and $\int_0^{\tau_0} \int_{\mathbb{R}} \Gamma(\tau,y) e^{py + q\tau}\,dyd\tau <\infty$ for all $(p,q) \in \mathbb{R}^2$.

\item[(H4)]
The functions $R_1(s)$ and $R_2(s)$, given by \eqref{eq:h4}, satisfy 
\begin{equation}\label{eq:H7}
R_i(s) = \underline{R}_i(s) \quad \text{ a.e. in } (0,\infty), \quad \text{ and }\quad 
\underline{R}_{1}(s) +  \underline{R}_{2}(s)>0 \quad \text{ for each }s > 0.
\end{equation}
and one of the following holds:
\begin{itemize}
    \item[{\rm (i)}] $R_1$ and $R_2$ are both non-increasing, or both non-decreasing;    
    \item[{\rm (ii)}] $R_1$ is continuous, and $R_2$ is monotone;    
    \item[{\rm (iii)}] $R_2$ is piecewise constant, and $R_1$ and $R_1+R_2$ are both locally monotone.
\end{itemize}

\item[(H5)] For any $\phi\in L^\infty([-\tau_0,0]\times\R)$ with $\phi\ge\not \equiv 0$ there exists $\underline{s}>0$ such that the solution $u$ of \eqref{Model} satisfies $$\liminf\limits_{t \to +\infty} \inf_{0 \leq x \leq \underline{s} t} u(t,x,\phi)>0.$$

\end{enumerate}
The hypothesis (H1) says that the nonlinearity is sublinear. In case $f_2$ is nontrivial, we only assume $f_2(\cdot,\cdot,u)$ is monotone close to 0, in other words, the full system might not admit the comparison principle;
(H2) is a self-limitation assumption; (H3) says that $\Gamma$ has finite moments to ensure a finite spreading speed; some sufficient conditions to guarantee uniqueness in the underlying HJ equation are given in (H4); (H5) means the population spreads successfully to the right.

\begin{remark}
Hypothesis (H5) can be guaranteed if $\liminf\limits_{t \to \infty}\left[ \inf\limits_{x \in [0,\underline{s} t]} \partial_u f_1(t,x,0)\right]>0.$ (See \cite{BHN2008}.) This is equivalent to $\underline{R}_1(s)>0$ for each $s \in [0,\underline{s}]$. 
More generally, if there exist $\underline{s}, c_1 >0$ such that
$\inf\limits_{c_1 < s<c_1+\underline{s}} \underline{R}_1(s)> (c_1)^2/4$, one can apply a change of coordinates $x' = x-c_1 t$, which introduces a drift term, then (H5) holds,  
so that the arguments of this paper can also be applied.
This is connected with the results in \cite{HSL2019,Hu2020,Li2014} when $\underline{R}_1(s) <0$ in $(-\infty,c_1]$ and $\underline{R}_1(s) >0$ in $(c_1,\infty)$.
\end{remark}
\begin{remark}
For example, take $\Gamma$ to be any probability kernel on $[0,\tau_0]\times\mathbb{R}$ with finite moments, and 
$$
f_1(t,x,u) = (r_1(x- c_1 t)-u)u \quad \text{ and }\quad f_2(t,x,v) = r_2(x-c_2 t) v e^{-v}.
$$
where $r_i$ are monotone functions such that $\inf r_1 + \inf r_2 >0$.
If $r_1,r_2$ are both increasing or both decreasing, then {\rm (H1)-(H5)} are satisfied with
$$
R_i(s) = \begin{cases}
r_i(-\infty)& \text{ for }s < c_i\\
r_i(-\infty) \vee r_i(+\infty) & \text{ for }s = c_i\\
r_i(+\infty) & \text{ for }s > c_i\\
\end{cases}
\quad \text{ and }\quad 
\underline{R}_i(s) = \begin{cases}
r_i(-\infty)& \text{ for }s < c_i\\
r_i(-\infty) \wedge r_i(+\infty) & \text{ for }s = c_i\\
r_i(+\infty) & \text{ for }s > c_i.\\
\end{cases}
$$
\end{remark}
\begin{remark}\label{rmk:746}
When $f_2 \equiv 0$, and $f_1(t,x,u)=(r(t,x)-u)u$, then {\rm(H1)-(H5)} reduces to 
\begin{itemize}
    \item[$(\tilde{\rm H})$] $R_1(s)$ is locally monotone, $\underline{R}_1(s) >0$ for each $s \geq 0$ and $R_1=\underline{R}_1$ a.e.
\end{itemize}
where $R_1(s) = \limsup\limits_{t\to\infty \atop s' \to s}r(t,s't)$ and $\underline{R}_1(s) = \liminf\limits_{t\to\infty \atop s' \to s}r(t,s't)$.
\end{remark}

As observed in \cite{LLL_2020c}, the exact spreading speed can be characterized in terms of the following Hamilton-Jacobi equation on a one-dimensional domain:
\begin{equation}\label{rhoeq1''}
\min\left\{\rho,\rho-s\rho'+|\rho'|^2+R_1(s)+R_2(s)\int_{0}^{\tau_0}\int_{\R }\Gamma(\tau, y)e^{(\rho-s\rho')\tau+\rho' y} dy d\tau\right\}=0 \quad \text{ in }(0,\infty). 
\end{equation}
To tackle the classical issue of uniqueness for the Hamilton-Jacobi equation \eqref{rhoeq1''}, 
we need one of the following to hold for the function $R_2$, and sometimes $R_1$, as $s\to +\infty$.  
\begin{itemize}
    \item[{\rm(H6)}] $R_2(s)$ is identically zero, or  
    non-increasing in $[s_0,\infty)$ for some $s_0$,
    \item[{\rm(H6$'$)}] $\lim\limits_{s\to+\infty}R_2(s)$ exists and is positive, and  $\displaystyle R_2(s) \geq  \left(\sup_{(s,\infty)}R_2\right) - o\left(\frac{1}{s^2}\right)$ for $s \gg 1$, 
       \item[{\rm(H6$''$)}] One of $\lim\limits_{s \to +\infty}R_i(s)$ exists, $\limsup\limits_{s\to\infty} R_2(s)>0$, and $\Gamma(\tau,y) = 0$ in $[0,\tau_1]\times (-\infty,0),$ for some  $\tau_1 \in (0,\tau_0]$.
\end{itemize}

Next, we specify the spreading speeds $\hat{s}^\mu$ as a free boundary point of the solution of \eqref{rhoeq1''}, which depends implicitly on $R_1(s)$, $R_2(s)$, $\Gamma(\tau,y)$ and $\mu \in (0,\infty]$. For the the definition of viscosity super- and sub-solutions, see Section \ref{sec:2}.
\begin{proposition}\label{prop:1.2}
Let $R_1 \in L^\infty(0,\infty)$ be locally monotone, $R_2 \in L^\infty(0,\infty)$ non-negative, and either monotone or piecewise constant. Suppose either
$$
{\rm(a)}\quad \mu \in (0,\infty), \quad \text{ or }\quad {\rm(b)} \quad \mu = \infty \, \text{ and one of {\rm(H6), (H6$'$) or (H6$''$)} holds,}
$$
then there exists a unique viscosity solution $\hat\rho^\mu \in C([0,\infty))$ of \eqref{rhoeq1''} such that
    \begin{equation}\label{eq:hjbc}
    \hat\rho^\mu(0) = 0, \quad \text{ and }\quad \lim_{s \to +\infty} \frac{\hat\rho^\mu(s)}{s} = \mu.
    \end{equation}
    Furthermore, $s\mapsto \hat\rho^\mu(s)$ is non-decreasing in $[0,\infty)$, so that the free boundary point
    \begin{equation}\label{eq:snlp}
    \hat{s}^\mu:= \sup \{ s : \hat\rho^\mu(s) = 0\} 
    \end{equation}
    is well-defined.
\end{proposition}

\begin{theorem}\label{thm:1.3}
Assume {\rm (H1)-(H5)}.
    Let $u$ be a solution to \eqref{Model} with initial data satisfying (IC$_\mu$) for some $\mu \in (0,\infty)$, then
    \begin{equation}\label{eq:1.5}
        \begin{cases}
        \lim\limits_{t\to\infty} \sup\limits_{x \geq (\hat{s}^\mu) + \eta)t} u(t,x) = 0 & \text{ for each }\eta >0,\\
        \liminf\limits_{t\to\infty} \inf\limits_{0\leq x \leq (\hat{s}^\mu - \eta)t} u(t,x) >0 & \text{ for each }0 < \eta < \hat{s}^\mu,\\
        \end{cases}
    \end{equation}
    where $\hat{s}^\mu$ is given in \eqref{eq:snlp}.
\end{theorem}

\begin{theorem}\label{thm:1.3b}
Assume {\rm (H1)-(H5)}. Let $u$ be a solution to \eqref{Model} with initial data satisfying (IC$_\infty$). If one of the conditions {\rm(H6)}, {\rm(H6$'$)} or {\rm(H6$''$)} holds, then
    \begin{equation}\label{eq:1.5'}
        \begin{cases}
        \lim\limits_{t\to\infty} \sup\limits_{x \geq (\hat{s}^\infty) + \eta)t} u(t,x) = 0 & \text{ for each }\eta >0,\\
        \liminf\limits_{t\to\infty} \inf\limits_{0\leq x \leq (\hat{s}^\infty - \eta)t} u(t,x) >0 & \text{ for each }0 < \eta < \hat{s}^\infty,\\
        \end{cases}
    \end{equation}
    where $\hat{s}^\infty$ is given in \eqref{eq:snlp} with $\mu=\infty$.
\end{theorem}
\begin{remark}
Since we do not impose assumptions on $f_i(t,x,u)$ for large $u$, except that it eventually becomes negative in (H2), the convergence of $u$ to a homogeneous equilibrium does not hold in general. However, if we strengthen {\rm(H2)} to
\begin{itemize}
    \item[{\rm(H2$'$)}] For $i=1,2$, $f_i(t,x,1)\equiv 0$, and $(u-1)f_i(t,x,u) <0$ for all $(t,x,u)$ such that $u \neq 1$.
    \end{itemize}
Then one can argue as in \cite{XuXiao_2016} that $u(t,x) \to 1$ for $x < \hat{s}^\mu t$ and $t \gg 1$.
\end{remark}
\begin{remark}
To compare our approach with that of Berestycki and Nadin, we only homogenize along the ray $x/t=s$ for each $s$ here, while in \cite{Nadin2012,Nadin2020}, the information in $\{(t,ye): t \gg 1, x \gg 1\}$ for each direction $e$ is homogenized via the notion of principal eigenvalues of the parabolic problems. In \cite{Nadin2020}, it is demonstrated that in higher dimensions, sometimes the spreading speed in direction $e$ does not depend only on what happens in the $e$ direction. The same can be observed in a shifting habitat on $\mathbb{R}$, where the spreading speed is nonlocally determined; see Theorems \ref{thm:4} and \ref{thm:kppp}.
\end{remark}

\subsection{Some explicit formulas of spreading speeds}

The following theorem concerns the spreading in ``asymptotically homogeneous" environments, and generalizes the spreading results of \cite{Gourley2003,XuXiao_2016}.  For simplicity, we assume that $\Gamma(\tau,y)$ is symmetric in the variable $y$ for each $\tau\in[0,\tau_0]$ in the next theorems.

\begin{theorem}\label{thm:nonkpp1}    Let $u$ be a solution to \eqref{Model} with initial data satisfying (IC$_\mu$) for some $\mu \in (0,\infty)$. Assume {\rm (H1)-(H5)} and, in addition,
for $i=1,2$, there are positive constants $r_2$, $r_1+r_2$ such that 
\begin{equation}\label{eq:R}
    R_i(s)=r_i\, \text{ for every } s\in\R,\quad \underline{R}_i(s)=r_i\, \text{ for almost every (a.e.) } s\in \R.
\end{equation}
    Then \eqref{eq:1.5} holds, and the spreading speed $\hat{s}^\mu$ is given by
    \begin{equation}\label{eq:R22}
    \hat{s}^\mu = \begin{cases}
\inf\limits_{p>0}\frac{\lambda(p)}{p} & \text{ if } \mu\in[\mu^*,\infty),\\
\frac{\lambda(\mu)}{\mu} & \text{ if } \mu\in(0,\mu^*),
    \end{cases}
    \end{equation}
where  $\lambda(p):\mathbb{R}\to (0,\infty) $ is uniquely defined by the implicit formula\begin{equation}\label{eq:lambdap} 
\Delta(\lambda,p):=-\lambda+p^2+r_1+ r_2\int_{0}^{\tau_0}\int_{\R }\Gamma(\tau,y) e^{p y-\lambda \tau}  dy d\tau=0.
\end{equation} and
 $\mu^*>0$ such that  $\frac{\lambda(\mu^*)}{\mu^*}=\inf\limits_{p>0}\frac{\lambda(p)}{p}>0$.
\end{theorem}
\begin{remark} A sufficient condition of {\rm(H4)}-(H5) and \eqref{eq:R} is when $f_i$ are independent of $t,x$ such that $f'_i(0)=r_i$. Moreover, the homogeneous coefficients case could be extended to a large class of space-time heterogeneous problems, e.g.,  $\sum^2\limits_{i=1}\inf\limits_{[0,\infty)\times\mathbb{R}} \partial_uf_i(t,x,0) >0$, and 
$$
r_1-\sum_{i=1}^m k_i(x - c_it^{\alpha_i}) \leq \partial_uf_1(t,x,0) \leq r_1\quad \text{ and }\quad 0<r_2- \sum_{i=1}^m \tilde{k}_i(x- c_i t^{\alpha_i}) \leq \partial_uf_{2}(t,x,0) \leq r_2,
$$
where $m \in \mathbb{N}$, $\alpha_i,c_i$, $r_2$, $r_1+r_2$  are positive constants, 
and $k_i(\cdot)$ and $\tilde{k}_i(\cdot)$ are non-negative functions that are compactly supported on $\R$.
\end{remark}
Next, we turn our attention to environments with one shift. 
Let $R_{i,-},R_{i,+}$ be fixed constants. For $p \in \mathbb{R}$, define $\lambda_{-}(p)$ and $\lambda_+(p)$ implicitly by 
$$
0 = \Delta_{\pm}(\lambda,p) = -\lambda + p^2 +R_{1,\pm}+ R_{2,\pm}\int_0^{\tau_0} \int_\mathbb{R} \Gamma(t,y) e^{p y - \lambda \tau}\,dyd\tau, 
$$
Define $(c^*_-,\mu^*_-)$ and $(c^*_+,\mu^*_+)$ by
$$
0<c^*_{\pm} =  \inf_{p >0} \frac{\lambda_{\pm}(p)}{p} = \frac{\lambda_1(\mu^*_\pm)}{\mu^*_\pm}. 
$$
\begin{theorem}\label{thm:4}
    Assume {\rm (H1)-(H5)} and, in addition, for $i=1,2$, there are constants $R_{i,-},R_{i,+}$ and $c_1>0$ such that
    \begin{equation}\label{eq:onesup}
    R_i(s)=\begin{cases}
    R_{i,+}  &\text{ if }s\geq c_1 \\
    R_{i,-}  &\text{ if }s< c_1 
    \end{cases}\, \text{ for every } s\in\R,\quad \underline{R}_i(s)=\begin{cases}
    R_{i,+}  &\text{ if }s> c_1 \\
    R_{i,-}  &\text{ if }s\leq c_1 
    \end{cases}\, \text{ for a.e. } s\in \R,   
    \end{equation}
    and
    $$
    R_{i,-} \leq R_{i,+}\quad \text{ and }\quad R_{1,+}+R_{2,+}>R_{1,-}+R_{2,-}>0.
    $$
    Then \eqref{eq:1.5} holds, and  the rightward spreading speed $\hat{s}^\mu$ can be given by (see Figure \ref{fig:cmu})

 (i) $\mu \in (0,\mu^*_+]$, then
 	\begin{equation}\label{eq:1.18}
\hat 	s^\mu(c_1)=\begin{cases}
\lambda_+(\mu)/\mu &\text{ if  } c_1\le \lambda_+(\mu)/\mu,\\
	\lambda_-(\underline p(c_1,\mu))/\underline p(c_1,\mu) & \text{ if } c_1>  \lambda_+(\mu)/\mu\text{ and } \underline p(c_1,\mu) <\mu^*_- \\
	c^*_-,& \text{otherwise}
	\end{cases}
	\end{equation}   
where    $\underline{p}(c_1,\mu)$ is the smallest root of
    \begin{equation}\label{eq:up'}
    c_1p-\lambda_-(p)=c_1\mu-\lambda_+(\mu).
    \end{equation}
    
(ii) $\mu \in (\mu^*_+, \infty)$, then    
 	\begin{equation}\label{eq:1.20}
\hat 	s^\mu(c_1)=\begin{cases}
c^*_+ &\text{ if  } c_1\le c^*_+,\\
	\lambda_-(\bar p(c_1))/\bar p(c_1) & \text{ if } \lambda'_+(\mu^*_+)<c_1\le \lambda'_+(\mu)  \text{ and } \bar p(c_1)< \mu^*_-,\\
	\lambda_-(\underline p)/\underline p & \text{ if } c_1>  \lambda'_+(\mu)\text{ and } \underline p(c_1,\mu) <\mu^*_-, \\
	c^*_-,& \text{otherwise}.
	\end{cases}
	\end{equation}  
		where  $\underline p(c_1,\mu)$ is the smallest root of \eqref{eq:up'} and $\bar p(c_1)$ is the smallest root of 
	\begin{equation}\label{eq:op'}
	c_1p-\lambda_-(p)=c_1\Psi_+(c_1)-\lambda_+(\Psi_+(c_1))
	\end{equation}
	In particular, if $\mu=\infty$, then
\begin{equation}\label{eq:1.22}
	\hat s^\infty(c_1)=\begin{cases} 
	c^*_+ & \text{ if } c_1\le c^*_+,\\
	\lambda_-(\bar p(c_1))/\bar p(c_1) & \text{ if } c_1\in(c^*_+,\bar c_1),\\
	c^*_-& \text{ if } c_1\ge\bar c_1.
	\end{cases}
	\end{equation}
	where
	$\bar c_1$ is the  unique positive number such that $\bar{p}(\bar c_1)=\mu^*_-$.
\end{theorem}

\begin{remark}
A sufficient condition of {\rm(H4)-(H5)} and \eqref{eq:onesup} is when there exists $c_1 >0$ such that  $f_i(t,x,u) = f_i(x-c_1t,u)$ satisfying  $$\begin{cases}R_{i,-}=\partial_uf_i(-\infty,0)\leq\partial_uf_i(y,0)\leq \partial_u f_i(+\infty,0)=R_{i,+}\quad \text{ holds for all } y\in\R,\\
R_{1,+}+ R_{2,+}> R_{1,-}+ R_{2,-}>0.\end{cases}$$
\end{remark}
\vspace{-1cm}
\begin{figure}[H]
	\begin{center}
	    \subfloat[]
		{	\label{fig}\includegraphics[height=45mm, width=70mm]{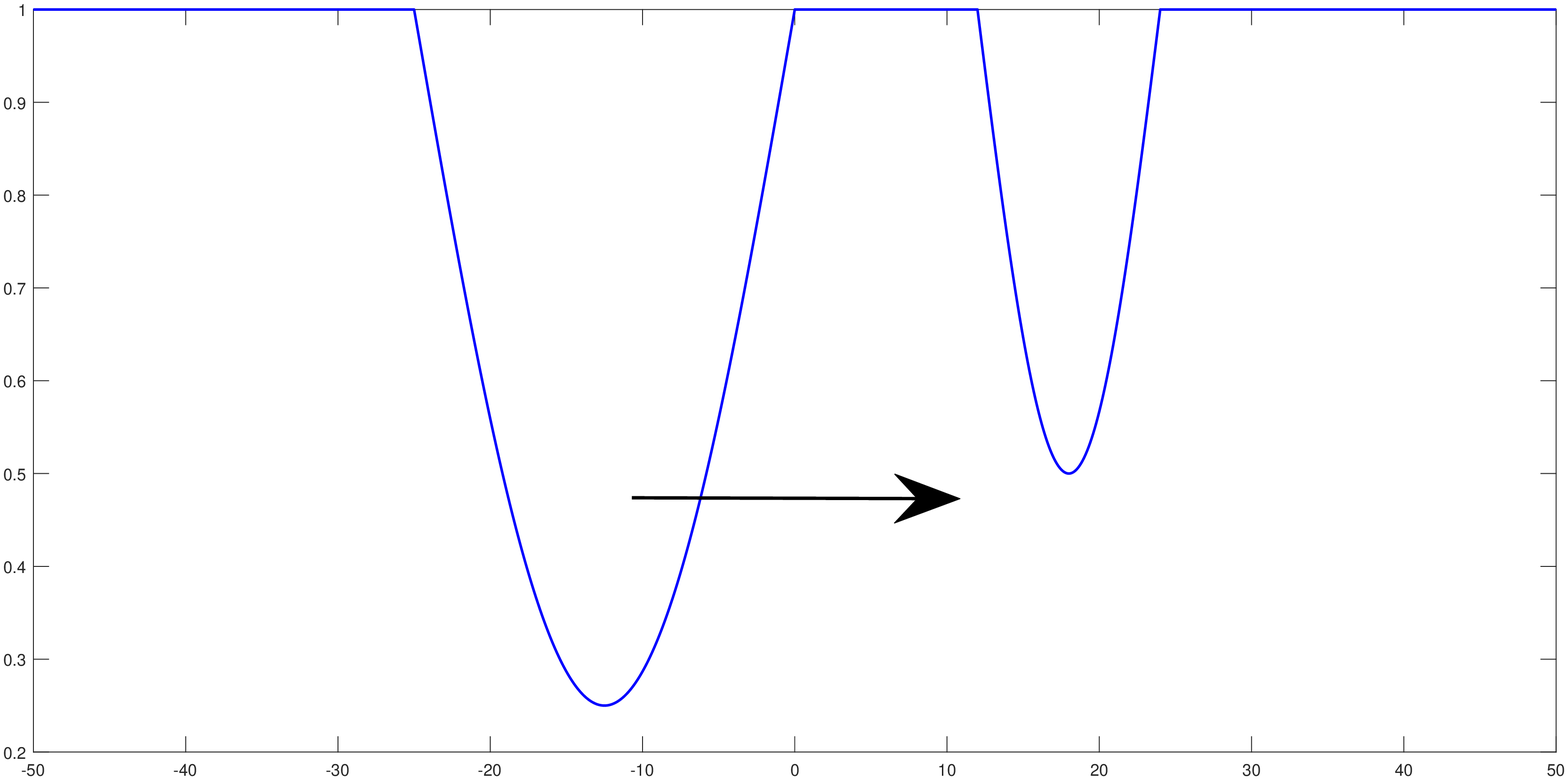}}
		\subfloat[]	{\label{fig2}\includegraphics[height=45mm, width=70mm]{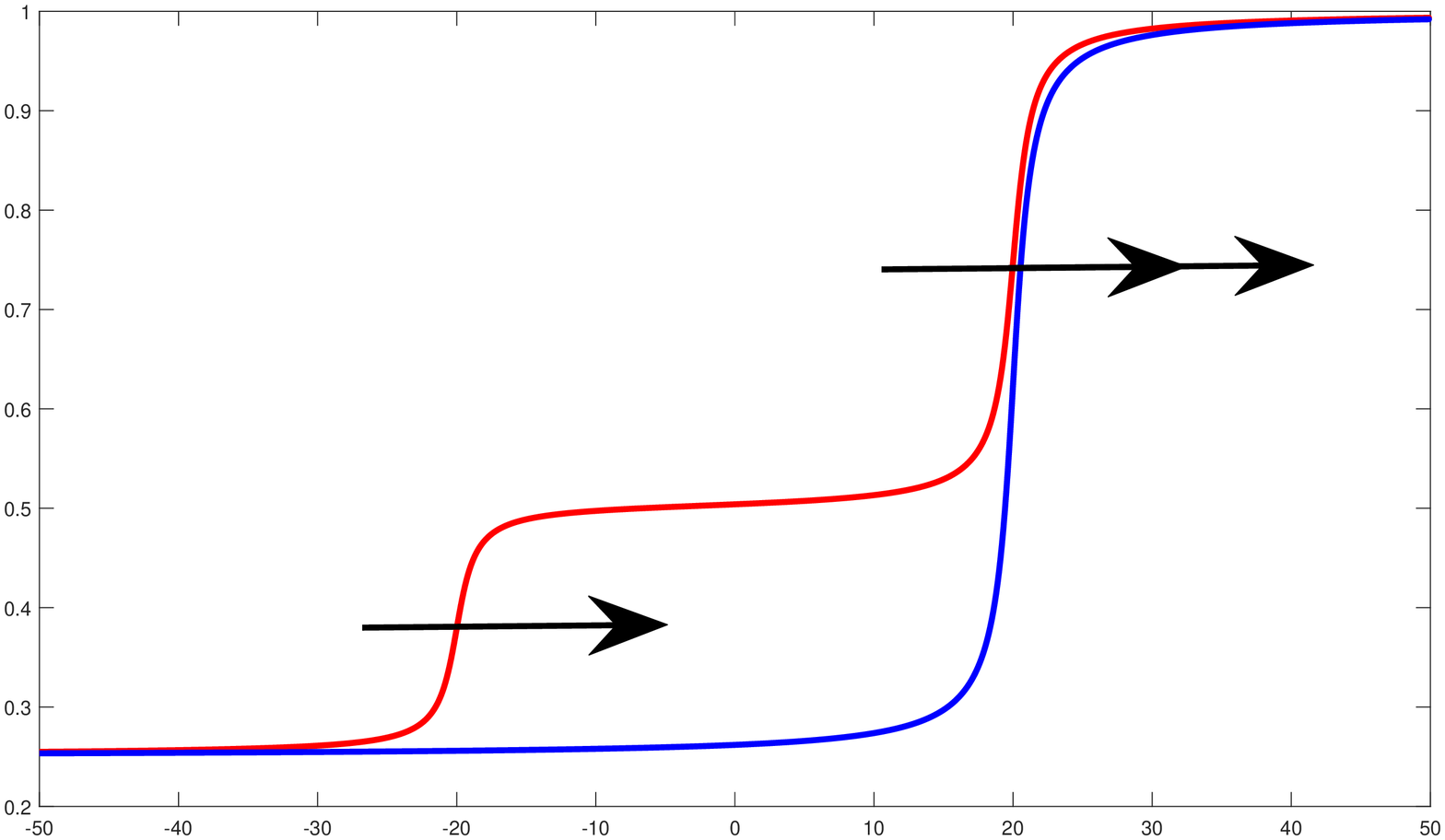}}
	\end{center}
			\vspace{-0.5cm}
{\small 	\caption[]{In the left panel, $r(t,x)=r_0-r_1(x-c_1t)$ for a compactly supported, non-negative and bounded function $r_1$ with $r_0,c_1>0$ where Theorem \ref{thm:kpp1} applies. In the right panel, blue curves represents $r(t,x)=\tilde{r}(x-c_1t)$ for a non-decreasing function $\tilde{r}$ with	$1/4=\tilde{r}(-\infty)< \tilde{r}(+\infty)=1$ where Theorem \ref{thm:kppp} applies; red curves represents $r(t,x)=\alpha\tilde{r}(x-c_1t)+(1-\alpha)\tilde{r}(x-c_2t)$ for $c_1>c_2>0$ and $\alpha\in(0,1/2)$, where Theorem \ref{thm:kppp'} applies.} }
\end{figure} 
\begin{figure}[H]
    \centering
  \includegraphics[height=56mm, width=72mm]{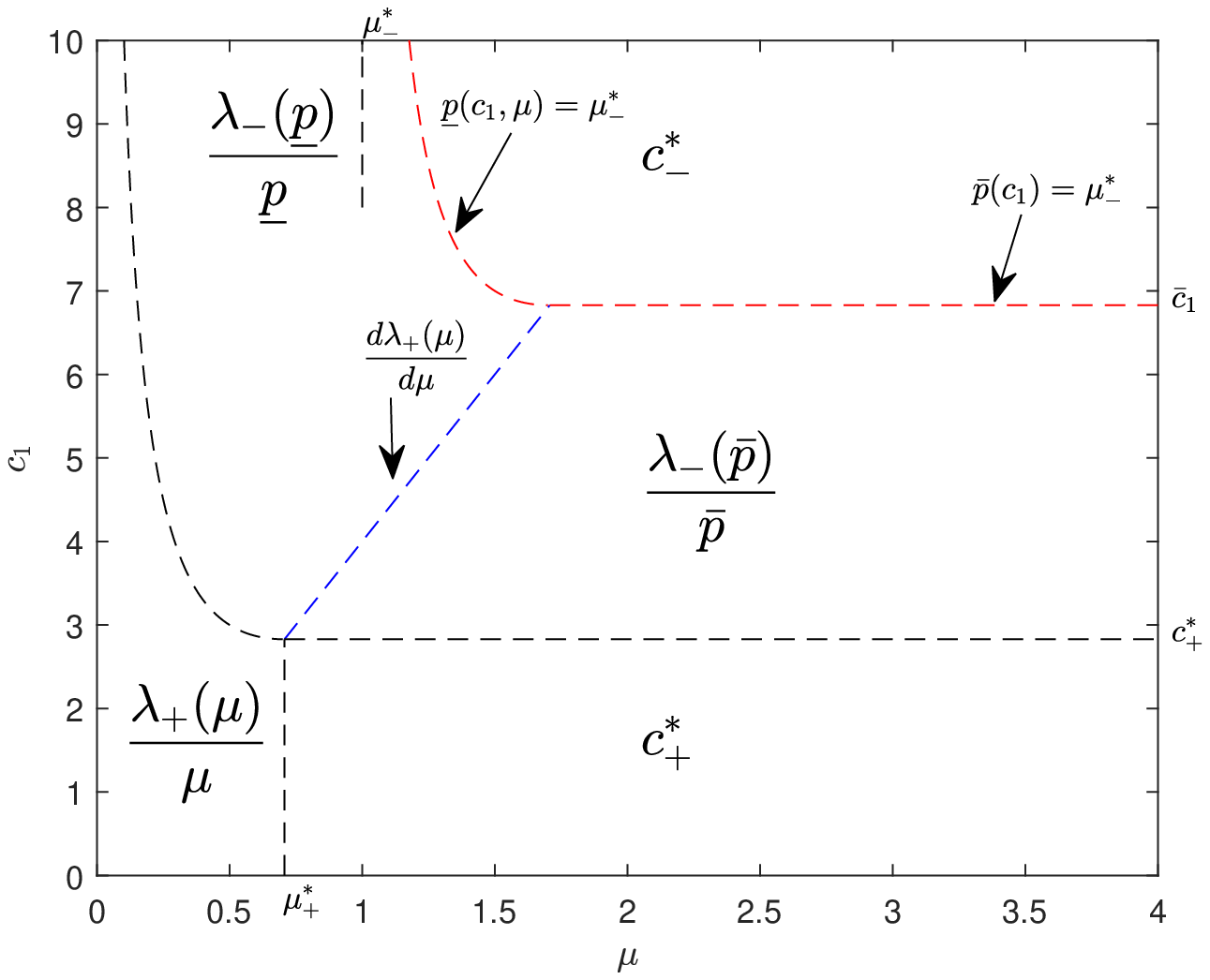}
{\small\caption{\label{fig:cmu}Parameter regions in $(\mu,c_1)$-plane corresponding to various cases of Theorem \ref{thm:4}.}}
  
\end{figure}

\subsection{Applications to the Fisher-KPP model}
In this subsection, we state our new results for the Fisher-KPP equation \eqref{eq:model4},
which could be easily derived from the results in last section or the arguments in Section 3.
Throughout this subsection, we impose the following assumption on $r$.

\noindent{\rm(F)} $r\in L^\infty(\R\times\R)$ and $\liminf\limits_{M \to \infty} \left[ \inf\limits_{[M,\infty)\times[M,\infty)} r(t,x)\right] >0$.

First, we show a new spreading result concerning a class of asymptotically homogeneous environment, for which there is a constant $r_0>0$ such that 
\begin{equation}\label{eq:F1}
\limsup\limits_{t \to \infty, \,s' \to s} r\left(t,s't \right) = r_0 \,\, \text{for every }s \in (0,\infty),\\
\liminf\limits_{t \to \infty,\, s' \to s} r\left(t,s't \right) = r_0 \,\,\text{ for a.e. }s \in (0,\infty).
%
\end{equation}
An example is $r(t,x) = r_0 - \int_0^\infty r_1(x-st) d\mu(s)$, where $r_1$ is a non-negative function with compact support, and $\mu$ is $\sigma$-finite measure on $[0,\infty)$. 
See Figure \eqref{fig} for the prototypical example $r(t,x) = r_0 - r_1(x-c_1t)$.

\begin{theorem}\label{thm:kpp1}
Consider the Cauchy problem \eqref{eq:model4}, with initial data $u_0(x)$ satisfying {\rm(IC$_\mu$)} for some $\mu \in (0,+\infty]$, respectively.
If $r(t,x)$ satisfies {\rm(F)}, and also
\eqref{eq:F1} for some constant $r_0>0$, 
then the spreading speed $\hat s^\mu$ is given by \eqref{eq:snlp}. Furthermore, it can be given explicitly as follows
\begin{equation}\label{eq:1.24b}
\hat s^\mu = \mu + \frac{r_0}{\mu} \quad \text{ if }\mu \in (0,\sqrt{r_0}),\quad \text{ and }\quad  \hat s^\mu  =  2\sqrt{r_0} \quad \text{ if } \mu \in [\sqrt{r_0},\infty].
\end{equation}
\end{theorem}
\begin{remark}
In \eqref{eq:F1}, the convergence of limit superior ``everywhere" cannot be relaxed to ``almost everywhere", because it is possible for locked waves to form, i.e. $c^* = \max\{2\sqrt{r_0},c_1\}$ \cite{Scheel_2014}. 

On the other hand,  if the condition  ``almost everywhere" in the convergence of limit inferior is strengthened to ``everywhere", the spreading result is proved in
\cite[Proposition 3.1]{Nadin2020}. 
\end{remark}
More generally, we observe that (H1)-(H6) are satisfied for the classical KPP equation \eqref{eq:model4} when
$$
r(t,x) = \sum_{i=1}^m r_i(x-c_1 t)
$$
such that $c_i \in \mathbb{R}$ are distinct, and for each $i$,    
$r_i \in C(\mathbb{R})$ is strictly positive and monotone; see Remark \ref{rmk:746}.  Hence, the spreading speed can be characterized by the free-boundary problem \eqref{eq:snlp}. 
Below, we completely work out the spreading speed when there is only one shifting speed.
\begin{theorem}\label{thm:kppp}
Consider the Cauchy problem \eqref{eq:model4}, with initial data $u_0(x)$ satisfying {\rm(IC$_\mu$)} for some $\mu \in (0,\infty]$, respectively. If $r(t,x)$ satisfies {\rm(F)}, and there are constants $r_2 > r_1 >0$ such that
$$
\limsup_{\epsilon\to0
	\atop(t',x')\to(t,x)} r\left( \frac{t'}{\epsilon}, \frac{x'}{\epsilon}\right) = \begin{cases}
	r_2  &\text{ if }x \geq c_1 t,\\
	r_1 &\text{ if }x < c_1 t.
	\end{cases}\,\,\text{ for all }(t,x) \in (0,\infty)\times\mathbb{R},
$$
and 
$$
\liminf_{\epsilon\to0
	\atop(t',x')\to(t,x)} r\left( \frac{t'}{\epsilon}, \frac{x'}{\epsilon}\right) = \begin{cases}
	r_2  &\text{ if }x > c_1 t,\\
	r_1 &\text{ if }x \leq c_1 t.
	\end{cases}\,\,\text{ for a.e. }(t,x) \in (0,\infty)\times\mathbb{R}.
$$
(e.g. $r(t,x) = \tilde{r}(x-c_1t)$ for some increasing function $\tilde r$ such that $\tilde{r}(-\infty) = r_1$ and $\tilde{r}(\infty) = r_2$, see figure \eqref{fig2}.) 
Then the spreading speed $\hat s^\mu$ is given by \eqref{eq:snlp}. Furthermore, it can be given explicitly as follows.

%
 %
  (i) $\mu\in(0, \sqrt{r_1}]$
\begin{equation}\label{123}
	\hat s^\mu(c_1)=\begin{cases} 
		\mu+\frac{r_2}{\mu} & \text{ if } c_1\in(-\infty, \mu+\frac{r_2}{\mu} ],\\
	\frac{c_1-\sqrt{(c_1-2\mu)^2+4(r_2-r_1)}}{2}+\frac{2r_1}{c_1-\sqrt{(c_1-2\mu)^2+4(r_2-r_1)}} & \text{ if } c_1\in(\mu+\frac{r_2}{\mu},+\infty),
	\end{cases}
\end{equation}
(ii) $\mu\in(\sqrt{r_1},\sqrt{r_2})$
\begin{equation}\label{1234}
	\hat s^\mu(c_1)=\begin{cases} 
		\mu+\frac{r_2}{\mu} & \text{ if } c_1\in(-\infty, \mu+\frac{r_2}{\mu} ],\\
	\frac{c_1-\sqrt{(c_1-2\mu)^2+4(r_2-r_1)}}{2}+\frac{2r_1}{c_1-\sqrt{(c_1-2\mu)^2+4(r_2-r_1)}} & \text{ if } c_1\in(\mu+\frac{r_2}{\mu},\frac{\mu^2+r_2-2r_1}{\mu-\sqrt{r_1}}),\\
	2\sqrt{r_1} & \text{ if } c_1\in[\frac{\mu^2+r_2-2r_1}{\mu-\sqrt{r_1}}.+\infty),
	\end{cases}
\end{equation}
(iii) $\mu\in[\sqrt{r_2},\sqrt{r_1}+\sqrt{r_2-r_1})$
\begin{equation}\label{12345}
	\hat s^\mu(c_1)=\begin{cases} 
	2\sqrt{r_2} & \text{ if } c_1\in(-\infty, 2\sqrt{r_2} ],\\
		\frac{c_1}{2}-\sqrt{r_2-r_1}+\frac{r_1}{\frac{c_1}{2}-\sqrt{r_2-r_1}}& \text{ if } c_1\in(2\sqrt{r_2},2\mu],\\
	\frac{c_1-\sqrt{(c_1-2\mu)^2+4(r_2-r_1)}}{2}+\frac{2r_1}{c_1-\sqrt{(c_1-2\mu)^2+4(r_2-r_1)}} & \text{ if } c_1\in(2\mu,\frac{\mu^2+r_2-2r_1}{\mu-\sqrt{r_1}}),\\
	2\sqrt{r_1} & \text{ if } c_1\in[\frac{\mu^2+r_2-2r_1}{\mu-\sqrt{r_1}}.+\infty),
	\end{cases}
\end{equation}
(iv) $\mu\in[\sqrt{r_1}+\sqrt{r_2-r_1},\infty]$
\begin{equation}\label{12}
	\hat s^\mu(c_1)=\begin{cases} 
		2\sqrt{r_2} & \text{ if } c_1\in(-\infty, 2\sqrt{r_2}],\\
		\frac{c_1}{2}-\sqrt{r_2-r_1}+\frac{r_1}{\frac{c_1}{2}-\sqrt{r_2-r_1}} & \text{ if } c_1\in(2\sqrt{r_2},2\sqrt{r_1}+2\sqrt{r_2-r_1}),\\
		2\sqrt{r_1}& \text{ if } c_1\in[2\sqrt{r_1}+2\sqrt{r_2-r_1},\infty).
	\end{cases}
\end{equation}
\end{theorem}

Before we close, we also state the spreading speed when there are two shifts in the environment, which makes use of a previous result regarding a special case of \eqref{rhoeq1''} when $R_2\equiv0$ and $R_1$ is piecewise constant. (For the derivation, see Theorem C in \href{https://arxiv.org/pdf/1910.04217.pdf}
{https://arxiv.org/pdf/1910.04217.pdf}.)
\begin{theorem}\label{thm:kppp'}
Consider \eqref{eq:model4}, with initial data $u_0(x)$ satisfying {\rm(IC$_\infty$)}. Suppose $r(t,x)$ satisfies {\rm(F)}, and there are positive constants $1>r_2 > r_1 >0$ and $c_1 > c_2 > 0$ such that
$$
\limsup_{\epsilon\to0
	\atop(t',x')\to(t,x)} r\left( \frac{t'}{\epsilon}, \frac{x'}{\epsilon}\right) = \begin{cases}
	1   &\text{ if } x \geq c_1 t,\\
	r_2  &\text{ if }c_2 t \leq x < c_1 t,\\
	r_1 &\text{ if }x < c_2 t.
	\end{cases}\,\,\text{ for every }(t,x) \in (0,\infty)\times\mathbb{R},
$$
and 
$$
\liminf_{\epsilon\to0
	\atop(t',x')\to(t,x)} r\left( \frac{t'}{\epsilon}, \frac{x'}{\epsilon}\right) = \begin{cases}
	1     &\text{ if } x > c_1 t,\\
	r_2  &\text{ if }c_2 t <x \leq  c_1 t,\\
	r_1 &\text{ if }x \leq c_2 t.
	\end{cases}\,\,\text{ for a.e. }(t,x) \in (0,\infty)\times\mathbb{R}.
$$
(e.g. $r(t,x) = \frac{1-r_2}{1-r_1}\tilde{r}(x-c_1t) + \frac{r_2-r_1}{1-r_1}\tilde{r}(x-c_2t)$ for some increasing function $\tilde r$ such that $\tilde{r}(-\infty) = r_1$ and $\tilde{r}(\infty) =1$, see Figure \eqref{fig2}.) 
Then the spreading speed $\hat s^\mu$ is given by \eqref{eq:snlp}. Furthermore, it can be given explicitly as follows.
\begin{equation}
\hat s^\infty(c_1,c_2)=\begin{cases}
2 \,&\text{ for } c_1\le2\\
\mu+\frac{r_2}{\mu} \, &\text{ for } c_1>2, \mu<\sqrt{r_2} \text{ and } c_2\le \mu+\frac{r_2}{\mu}\\
2\sqrt{r_2} \, &\text{ for } \mu\ge \sqrt{r_2} \text{ and }  c_2\le 2\sqrt{r_2}\\
\bar p+\frac{r_1}{\bar p} \, &\begin{cases}\text{ for } c_1>2, \mu<\sqrt{r_2} ,\bar p<\sqrt{r_1} \text{ and } c_2>\mu+\frac{r_2}{\mu},\\
\text{ for } \mu\ge\sqrt{r_2}, c_2<2\mu,\bar p<\sqrt{r_1} \text{ and } c_2>2\sqrt{r_2}
\end{cases} \\
\underline p+\frac{r_1}{\underline p} \, &\text{ for } \mu\ge \sqrt{r_2}, c_2\ge 2\mu,\underline p<\sqrt{r_1} \text{ and } c_2>2\sqrt{r_2}, \\
2\sqrt{r_1} \, &\text{ otherwise, }
\end{cases}
\end{equation}
where \begin{equation}\mu=\frac{c_1}{2}-\sqrt{1-r_2},\quad \bar p=\frac{c_2}{2}-\sqrt{(\frac{c_2}{2}-\mu)^2+r_2-r_1},\quad\underline{p}=\frac{c_2}{2}-\sqrt{r_2-r_1}.
\end{equation}
\end{theorem}

\section{Proof of Theorems \ref{thm:1.3} and \ref{thm:1.3b}}\label{sec:2}

\subsection{Outline of the main arguments}

Let $u(t,x)$ be the unique solution of \eqref{Model} with initial data $\phi$ satisfying (IC$_\mu$) for some $\mu \in (0,\infty]$. 
To analyze the spreading behavior of $u(t,x)$, we
introduce the large time and large space scaling parameter $\ep$
\begin{equation}\label{eq:scaling}
u^\ep(t,x) = u\left(\frac{t}{\ep},\frac{x}{\ep}\right),
\end{equation}
and relate the limit of $u^\ep$ as $\ep \to 0$ to two Hamilton-Jacobi equations. The first one is time-dependent:
    \begin{equation}\label{whjEq1'}
\min\left\{w,H(x/t,\partial_t w, \partial_x w)\right\}=0 \quad \text{ for  }(t,x) \in (0,\infty)\times (0,\infty), 
\end{equation}
and the second one is time-independent:
\begin{equation}\label{rhoeq1'}
\min\left\{\rho, H(s,\rho-s\rho',\rho')\right\} = 0 \quad \text{ for }s \in (0,\infty),
\end{equation}
where
$$
H(s,q,p) = q + p^2 + R_1(s) + R_2(s) \int_0^{\tau_0}\int_\mathbb{R} \Gamma(\tau,y)e^{q\tau + py}\,dyd\tau.
$$
Note that $H(s,q,p)$ is u.s.c., since the functions $R_1(s)$ and $R_2(s)$ are u.s.c by (H4). 
\begin{remark}\label{rmk:Husc}
We choose to work with \eqref{whjEq1'} since it is more explicitly connected with \eqref{Model} and \eqref{rhoeq1'}. It is possible to rewrite \eqref{whjEq1'} into a more standard form: $\min\{w, \partial_t w + \tilde{H}(x/t,\partial_x w) \}= 0$, where $\tilde{H}(s,p)$ is defined implicitly by 
$$
H(s,q,p) = 0 \quad \text{ if and only if }\quad -q = \tilde{H}(s,p),
$$
by exploiting the monotonicity of $H$ in $q$. See Appendix \ref{sec:A} for details. 
\end{remark}

Indeed, under the scaling of \eqref{eq:scaling}, the problem \eqref{Model} can be rewritten as
\begin{equation}\label{LinearModel2}
\begin{cases}
\partial_t u^\ep=\ep\partial_{xx} u^\ep +\frac{1}{\ep} [F_\ep(t,x,u^\ep) + f_{1,\ep}(t,x,u^\ep)], \, t>0, x \in \mathbb{R},\\
u^\ep(\theta,x)=\phi(\theta/\ep,x/\ep),\quad (\theta,x)\in[-\ep\tau,0]\times\R,
\end{cases}
\end{equation}
where $$F_\ep(t,x,u^\ep)=\int_{0}^{\tau_0}\int_{\mathbb{R}}\Gamma(\tau,y) f_{2,\ep}(t-\ep \tau,x-\ep y,u^\ep(t-\ep \tau,x-\ep y)) dy d\tau$$ with $f_{2,\ep}(t,x,v):=f_2(t/\ep,x/\ep,v)$ and $f_{1,\ep}(t,x,u):=f_1(t/\ep,x/\ep,u)$.

The proof of Theorem \ref{thm:1.3} relies on a WKB approach in front propagation, which was introduced in \cite{BCS90,ES89,Souganidis1997}. It is based on the real phase defined by the Hopf-Cole transform
$$
w^\ep(t,x)=-\ep \log{u^\ep(t,x)}.
$$
The equation of $w^\ep$ is
\begin{equation}\label{LinearModel3}
\begin{cases}
\partial_t w^\ep-\ep\partial_{xx} w^\ep+|\partial_{x} w^\ep|^2+\frac{F_\ep(t,x,u^\ep)+ f_{1,\ep}(t,x,u^\ep)}{u^\ep}=0, \quad & t>0, x>0,\\
w^\ep(\theta,x)=-\ep\log \phi(\theta/\ep,x/\ep), \quad & x\ge0, \theta\in[-\ep\tau,0], \\
w^\ep(t,0)=-\ep \log u^\ep(t,0), \quad &  t\ge-\ep\tau.
\end{cases}
\end{equation}

In the following, we apply the half-relaxed limit method, due to Barles and Perthame \cite{BP1987}, to pass to the (upper and lower) limits of $w^\ep$. 
More precisely, for each $(t,x)\in (0,\infty)\times\mathbb{R}$, we set
\begin{equation}\label{eq:wstar}
w^*(t,x)=\limsup_{\epsilon\to0
	\atop(t',x')\to(t,x)} w^\epsilon(t',x'),\quad w_*(t,x)=\liminf_{\epsilon\to0
	\atop(t',x')\to(t,x)} w^\epsilon(t',x').
\end{equation}
We will show that $w^*$ and $w_*$ are respectively viscosity sub- and super-solution of the time-dependent problem \eqref{whjEq1'} (see Proposition \ref{prop:2.3}). 
To pass to the (local uniform) limit, it is necessary to establish the uniqueness of viscosity solution for such problems so that $w^*=w_*$ and hence $w^\ep$ converges uniformly as $\ep \to 0$. If the initial data of $u$ has compact support, then the initial data of $w^*,w_*$ can be infinite and the uniqueness does not follow from standard PDE proofs. 
Previously, this mathematical issue was tackled in \cite{ES89} by way of a correspondence with the value function of a zero sum, two player  differential game with stopping times and the dynamics programming principle; see also \cite{CLS1989} for a method based on semigroup method. Our main novelty here is to provide an elementary proof of the uniqueness based on the 1-homogeneity of $w^*$ (resp. $w_*$): 
\begin{align}
w^*(t,x) &= \limsup_{\epsilon\to0
	\atop(t',x')\to(t,x)} -\epsilon \log u\left(\frac{t'}{\ep},\frac{x'}{\ep}\right)  \notag\\
	&=  t\limsup_{\epsilon\to0
	\atop(t'',x'')\to(1,x/t)} -(\epsilon/t) \log u\left(\frac{t''}{\ep/t},\frac{x''}{\ep/t}\right) = tw^*(1,x/t). \label{eq:rhooo2}
\end{align}

Hence, there exists $\rho^*(s)$ and $\rho_*(s)$ such that
\begin{equation}\label{eq:rhooo}
w^*(t,x) = t\rho^*(x/t) \quad \text{ and }\quad w_*(t,x) = t \rho_*(x/t).
\end{equation}
We will show that $\rho^*(s)$ and $\rho_*(s)$ are respectively the sub- and super-solution of the problem \eqref{rhoeq1'} in a one-dimensional domain (see Lemma \ref{lem:samehj}). 
By showing a novel comparison result for \eqref{rhoeq1'}, we have $\rho^*(s) = \rho_*(s)$ (see Proposition \ref{cp:rho}), so that they can be identified with the unique solution $\hat\rho^\mu$ of \eqref{rhoeq1''}, which satisfies (see Proposition \ref{prop:1.2})
$$
\hat\rho^\mu(s) = 0 \quad \text{ in }[0,\hat{s}^\mu],\quad \text{ and }\quad \hat\rho^\mu(s) >0 \quad \text{ in }(\hat{s}^\mu,\infty).
$$
Hence, $w^*(t,x) = w_*(t,x) = t \hat\rho^\mu(x/t)$, and we see that $w^\ep(t,x)$ converges in $C_{loc}((0,\infty) \times \mathbb{R})$, and
$$
\lim\limits_{\ep \to 0} w^\ep(t,x) = 0 \quad \text{ for }0 \leq x <  \hat{s}^\mu t,\quad \text{ and }\quad \lim\limits_{\ep \to 0} w^\ep(t,x)  >0 \quad \text{ for  } x > \hat{s}^\mu t.
$$
From this, the asymptotic behavior of $u^\ep(t,x)$ can then be inferred.

Let $w^*$ and $w_*$ be the half-relaxed limits as given by \eqref{eq:wstar}. The following lemma indicates that $w^*$ and $w_*$ are well-defined and finite-valued everywhere. 

\begin{proposition}\label{prop:2.1}
 Let $\phi$ satisfy $({IC}_\mu)$ for some $\mu\in(0,\infty]$ and $w^\ep$ be the solution of \eqref{LinearModel3}, then 
	\begin{equation}\label{eq:prop:2.1.0}
	0 \leq w_*(t,x) \leq w^*(t,x) <\infty \quad \text{ for each }(t,x) \in (0,\infty)\times\mathbb{R},
	\end{equation}
	and
	\begin{equation}\label{eq:prop:2.1.1}
	w^*(t,0)=w_*(t,0)=0 \, \text{ for }t>0,  \quad w^*(0,x)=w_*(0,x) =\begin{cases}
	\mu x & \text{ if } \mu\in(0,\infty),\\
	\infty& \text{ if } \mu=\infty.
	\end{cases}  \text{for }x >0. 
		\end{equation}
\end{proposition}
\begin{proof}
\noindent First, we claim that there exists $L >0$ such that $w^\ep(t,x) \geq  - \epsilon \log L$ in $[0,\infty)\times \mathbb{R}$.

Let $L_0$ be given in (H2). It suffices to choose $L\in [L_0,\infty)$ such that  
$0\leq u(t,x,\phi)\leq L$ in $[-\tau_0,0]\times \mathbb{R}$. Then (H2) and the maximum principle yield $u(t,x,\phi) \leq L$ for $(t,x) \in [0,\infty) \times \mathbb{R}.$ This proves that $w_*(t,x) \geq 0$ for all $(t,x)$. 

The first part of \eqref{eq:prop:2.1.1} follows from hypothesis (H5). The second part of \eqref{eq:prop:2.1.1} follows from Lemma \ref{lem:a4} in case $\mu \in (0,\infty)$, or Lemma \ref{lem:a5} in case $\mu = \infty$.

It remains to show the upper bound of \eqref{eq:prop:2.1.0}. By noting that
$$
\partial_t w^\ep - \ep \partial_{xx} w^\ep + |\partial_x w^\ep|^2 \leq C \quad \text{ for }(t,x) \in (0,\infty)\times(0,\infty),
$$
this follows from the proof of \cite[Lemma 3.2]{LLLDCDSA}.
\end{proof}

We recall the classical definition of discontinuous viscosity super- and sub-solutions to \eqref{rhoeq1'} following \cite{Barles2013}. See Definition \ref{def:AA'} for the corresponding definition for \eqref{whjEq1'}.
\begin{definition}
	We say that a lower semicontinuous function $\hat{\rho}$ is a viscosity super-solution of \eqref{rhoeq1'} if $\hat{\rho}\ge0$, and for all test functions $\phi\in C^1$, if $s_0$ is a strict local minimum of $\hat \rho-\phi$, then
	$$\hat\rho(s_0)-s_0\phi'(s_0)+|\phi'(s_0)|^2+R^*_1(s_0)+R^*_2(s_0)\int_{0}^{\tau_0}\int_{\R }\Gamma(\tau, y)e^{(\hat\rho(s_0)-s_0\phi'(s_0))\tau+\phi'(s_0) y} dy d\tau\ge0.$$
		We say that an upper semicontinuous function $\hat{\rho}$ is a viscosity sub-solution of \eqref{rhoeq1'} if for all test functions $\phi\in C^1$, if $s_0$ is a strict local minimum of $\hat \rho-\phi$ and $\hat{\rho}(s_0)>0$, then
	$$\hat\rho(s_0)-s_0\phi'(s_0)+|\phi'(s_0)|^2+R_{1,*}(s_0)+R_{2,*}(s_0)\int_{0}^{\tau_0}\int_{\R }\Gamma(\tau, y)e^{(\hat\rho(s_0)-s_0\phi'(s_0))\tau+\phi'(s_0) y} dy d\tau\le0.$$
	Finally, $\hat{\rho}$ is a viscosity solution of \eqref{rhoeq1'} if and only if $\hat \rho$ is a viscosity super- and sub-solution.
\end{definition}
The functions $R^*_i$ and $R_{i,*}$ appeared above denote respectively the upper semicontinuous (u.s.c) and lower semcontinuous (l.s.c) envelope of $R_i$, that is,
$$R^*_i(s)=\limsup_{s'\to s}R_i(s')\quad \text{ and }\quad R_{i,*}(s)=\liminf_{s'\to s}R_i(s').$$
\begin{remark}
If $R_i$ ($i=1,2$) satisfies {\rm(H4)}, then they are upper semi-continuous (u.s.c.) so we have $R_i(s) \equiv R^*_i(s)$ everywhere in $[0,\infty)$. 
\end{remark}

\begin{definition}\label{def:AA'}
	We say that a lower semicontinuous function $\hat{w}$ is a viscosity super-solution of \eqref{whjEq1'} if $\hat{w}\ge0$, and for all test functions $\phi\in C^1$, if $(t_0,x_0)$ is a strict local minimum of $\hat w-\phi$, then
	$$ H(x_0/t_0, \partial_t \phi(t_0,x_0), \partial_x \phi(t_0,x_0) \ge0.$$
		We say that an upper semicontinuous function $\hat{w}$ is a viscosity sub-solution of \eqref{whjEq1'} if for all test functions $\phi\in C^1$, if $(t_0,x_0)$ is a strict local minimum of $\hat w-\phi$ and $\hat{w}(t_0,x_0)>0$, then
	$$ H_*(x_0/t_0, \partial_t \phi(t_0,x_0), \partial_x \phi(t_0,x_0) \le0.$$
	Finally, $\hat{w}$ is a viscosity solution of \eqref{whjEq1'} if and only if $\hat{w}$ is a viscosity super- and sub-solution.
\end{definition}
In the above definition, $H_*(s,q,p) = q + p^2 + (R_1)_*(s) + (R_2)_*(s) \int_0^{\tau_0} \int_{\mathbb{R}} \Gamma(\tau,y) e^{q \tau + p y}\,dy d\tau$. We have also used the fact that $H$ is u.s.c..

Below, we relate the notion of viscosity super- and sub-solutions of \eqref{whjEq1'} and \eqref{rhoeq1'}.

\begin{lemma}\label{lem:samehj}
Suppose $w(t,x)$ and $\rho(s)$ are two functions such that
$$
w(t,x) = t\rho(x/t)\quad \text{ in }\Omega:=\{(t,x): 0 < x < c_b t\} \quad \text{ for some }c_b \in (0,\infty]. 
$$
Then 
$\rho(s)$ is a viscosity sub-solution (resp. super-solution) of \eqref{rhoeq1'} in the interval $(0,c_b)$ if and only if $w(t,x)$ is  a viscosity sub-solution (resp. super-solution) of \eqref{whjEq1'}
in $\Omega$.
\end{lemma}
\begin{proof}
The proof follows from a minor modification of that in \cite[Lemma 2.3]{LLL_2020c}. Below we only show the equivalence of viscosity sub-solutions. 

Let $\rho(s)$ be a viscosity sub-solution of \eqref{rhoeq1'} in $(0,c_b)$. We must  verify that $w(t,x)=t\rho\left(\frac{x}{t}\right)$ is a viscosity sub-solution of \eqref{whjEq1'}.  For any test function $\varphi\in C^1$, suppose that $w-\varphi$ attains a strict local maximum at point $(t_*,x_*)\in \Omega$ such that $w(t_*,x_*)>0$. Since $w(t,x)=t\rho\left(\frac{x}{t}\right)$ and $t_*>0$, we see that $\rho\left(\frac{x_*}{t_*}\right)>0$ and $f(y):=y \rho\left(\frac{x_*}{t_*}\right)-\frac{\varphi(yt_*,yx_*)}{t_*} $ admits a strict local maximum at $y=1$, so that letting $s_*=x_*/t_*$, we have
\begin{equation}\label{11}
\rho(s_*)-\partial_{t}\varphi(t_*,x_*)-s_*\partial_{x}\varphi(t_*,x_*)=0.
\end{equation}
{Next,} set $\phi(s):=\varphi(t_*,st_*)/t_*$. We observe that $\rho(s)-\phi(s)$ takes a strict local maximum point $s=s_*$ and $\rho(s_*)>0$. Note that $\phi'(s_*)=\partial_x\varphi(t_*,x_*)$, it follows from \eqref{11} that
$\partial_t\varphi(t_*,x_*)=\rho(s_*)-s_*\phi'(s_*).$
Hence at the point $(t_*,x_*)$, we have 
\begin{align*}
&\partial_t \varphi+|\partial_x \varphi|^2+R_{1,*}(x_*/t_*)+R_{2,*}(x_*/t_*)\int^{\tau_0}_0\int_{\R}\Gamma(\tau,y) e^{\tau\partial_t \varphi+y\partial_x \varphi } dy d\tau\\
&=\rho(s_*)\!-\!s_*\phi'(s_*)\!+\!|\phi'(s_*)|^2+R_{1,*}(s_*)+R_{2,*}(s_*)\int^{\tau_0}_0\int_{\R}\Gamma(\tau,y) e^{\tau (\rho(s_*)-s_*\phi'(s_*))+y\phi'(s_*)} dy d\tau\le0,
\end{align*}
where the last inequality holds since $\rho$ is a viscosity sub-solution of \eqref{rhoeq1'} with $\phi(s)$ being the test function. Therefore, $w$ is a viscosity sub-solution of \eqref{whjEq1'}.

Conversely, let $w(t,x)=t\rho\left(\frac{x}{t}\right)$ be a viscosity sub-solution of \eqref{whjEq1'} in $\Omega$. Choose any test function $\phi\in C^1$ such that $\rho(s)-\phi(s)$ attains a strict local maximum at $s_*$ and $\rho(s_*)>0$. Without loss of generality, we might assume $\rho(s_*)=\phi(s_*)$. Set $\varphi(t,x)=t\phi\left(\frac{x}{t}\right)+(t-1)^2$. It then follows that $w(t,x)-\varphi(t,x)=t(\rho(x/t)-\phi(x/t))-(t-1)^2$ attain a strict local maximum at $(1,s_*)$. Hence, by the definition of $w(t,x)$ being a sub-solution and the fact that $\partial_t\varphi(1,s_*)=\phi(s_*)-s_*\phi'(s_*)$ and $\partial_x\varphi(1,s_*)=\phi'(s_*)$, we infer that 
\begin{align*}
&\rho(s_*)-s_*\phi'(s_*)+|\phi'(s_*)|^2+R_{1,*}(s_*)+R_{2,*}(s_*)\int^{\tau_0}_0\int_{\R}\Gamma(\tau,y) e^{\tau (\rho(s_*)-s_*\phi'(s_*))+y\phi'(s_*)} dy d\tau\\
=&\partial_t \varphi(1,s_*)+|\partial_x \varphi(1,s_*)|^2+R_{1,*}(s_*)+R_{2,*}(s_*)\int^{\tau_0}_0\int_{\R}\Gamma(\tau,y) e^{\tau\partial_t \varphi+y\partial_x \varphi } dy d\tau
\le0,
\end{align*}
which implies $\rho$ is a sub-solution of \eqref{rhoeq1'}.
\end{proof}

In the following, we observe that the limit functions $w^*,w_*$ satisfies an equation without nonlocal term, even though the original problem \eqref{Model} has a nonlocal space/time delay. 
\begin{proposition}\label{prop:2.3}
Assume $u(t,x)$ is a solution of the nonlocal model \eqref{Model} with initial data satisfying (IC$_\mu$) for some $\mu \in (0,\infty]$. 
The functions  $w^*,w_*$, as given in Proposition \ref{prop:2.1}, are respectively viscosity supersolution and subsolution of \eqref{whjEq1'} in $(0,\infty)\times(0,\infty)$.
\end{proposition}
\begin{proof}
The proof essentially follows from a slight variation of \cite[Propositions 3.1 and 3.2]{BCS90}, we include it here only for the sake of completeness. First, we verify that $w_*$ is a viscosity super-solutions of \eqref{whjEq1'}. 
By \eqref{eq:prop:2.1.0}, $w_*(t,x)\ge 0$ on $(0,\infty)\times(0,\infty)$.

Fix a smooth test function $\phi$, without loss of generality, assume
$w_*-\phi$ has a strict global minimum at some point $(t_0,x_0)\in(0,\infty)\times(0,\infty)$. (We only need to check the strict global minima here, due to \cite[Prop. 3.1]{Barles2013}.)
It then suffices to show that $H(x/t,\partial_t \phi,\partial_x \phi)\ge0$ at $(t_0,x_0)$. (Here we used the fact that $H(s,q,p)$ is u.s.c. (Remark \ref{rmk:Husc}), so it coincides with its upper envelope $H^*(s,q,p)$.)
	
Clearly,
 there exists a sequence $\epsilon_n\to 0$ and a sequence of points $(t_n,x_n) \in (\ep_n \tau,\infty) \times (0,\infty)$ such that $w^{\ep_n}(t,x)-\phi$ has a global minimum at $(t_n,x_n)$ and that (see, e.g. \cite[Lemma 6.1]{Barles2013})
\begin{equation}
w^{\ep_n}(t_n,x_n) \to w_*(t_0,x_0) \quad \text{ as } n\to\infty \quad \text{ and }\quad  (t_n,x_n)\to (t_0,x_0)\quad \text{ as }n\to\infty.
\end{equation}
By definition of $(t_n,x_n)$ being the global minimum, 
\begin{equation}\label{eq:prop2.4a}
    w^{\ep_n}(t_n,x_n)-\phi(t_n,x_n)\le w^{\ep_n}(t,x)-\phi(t,x)\quad  \text{ for  }(t,x)\in(0,\infty)\times\R,
\end{equation}
it then follows from the maximum principle and {\rm(H1)} that at the point $(t_n,x_n)$
{\small\begin{align*}
0&\le \partial_t\phi(t_n,x_n)-\epsilon_n\partial_{xx}\phi(t_n,x_n)+|\partial_x \phi(t_n,x_n)|^2 + \frac{f_{1, \ep_n}(t_n,x_n,u^{\ep_n}(t_n,x_n))}{u^{\ep_n}(t_n,x_n)}\\ 
&\qquad+\int_{0}^{\tau_0}\int_{\mathbb{R}}\Gamma(\tau,y) \frac{f_{2, \ep_n}(t_n-\ep_n \tau,x_n-\ep_n y, u^{\ep_n}(t_n-\ep_n \tau,x_n-\ep_n y))}{u^{\ep_n}(t_n,x_n)}\,dyd\tau\\
&\le \partial_t\phi-\epsilon_n\partial_{xx}\phi+|\partial_x \phi|^2 + \partial_u f_{1, \ep_n}(t_n,x_n,0)\\ 
&\qquad +\int_{0}^{\tau_0}\int_{\mathbb{R}}\Gamma(\tau,y) \partial_uf_{2, \ep_n}(t_n-\ep_n \tau,x_n-\ep_n y,0) \frac{u^{\ep_n}(t_n-\ep_n \tau,x_n-\ep_n y)}{{u^{\ep_n}(t_n,x_n)}}\,dyd\tau\\
&\le \partial_t\phi-\epsilon_n\partial_{xx}\phi+|\partial_x \phi|^2 + \partial_u f_{1, \ep_n}(t_n,x_n,0)\\ 
&\qquad+\int_{0}^{\tau_0}\int_{\mathbb{R}}\Gamma(\tau,y) \partial_uf_{2, \ep_n}(t_n-\ep_n \tau,x_n-\ep_n y,0)e^{\frac{w^{\ep_n}(t_n,x_n)-w^{\ep_n}(t-\ep_n \tau,x-\ep_n y)}{\ep_n}} dy d\tau \\
&\le \partial_t\phi-\epsilon_n\partial_{xx}\phi+|\partial_x \phi|^2 + \partial_u f_{1, \ep_n}(t_n,x_n,0)\\ 
&\qquad+\int_{0}^{\tau_0}\int_{\mathbb{R}}\Gamma(\tau,y) \partial_uf_{2, \ep_n}(t_n-\ep_n \tau,x_n-\ep_n y,0)e^{\frac{\phi(t_n,x_n)-\phi(t_n-\ep_n \tau,x_n-\ep_n y)}{\ep_n}} dy d\tau\\
&\le\partial_t\phi-\epsilon_n\partial_{xx}\phi+|\partial_x \phi|^2 + \partial_u f_{1, \ep_n}(t_n,x_n,0)\\ 
&\qquad+\int_{0}^{\tau_0}\int_{\mathbb{R}}\Gamma(\tau,y) \partial_uf_{2, \ep_n}(t_n-\ep_n \tau,x_n-\ep_n y,0)e^{\partial_t\phi(\xi_n,\eta_n)s+\partial_x\phi(\xi_n,\eta_n)y} dy d\tau,
\end{align*}}
where $\xi_n\in(t_n-\ep_n \tau,t_n)$ and $\eta_n\in(x_n-\ep_n y, x_n)$. We used \eqref{eq:prop2.4a} in the fourth inequality. 
Letting $n\to \infty$, by Lebsegue dominated convergence theorem,  we get
$$0\le \partial_t\phi(t_0,x_0)+|\partial_x \phi(t_0,x_0)|^2+R_{1}(x_0/t_0)+R_{2}(x_0/t_0)\int_{0}^{\tau_0}\int_{\mathbb{R}}\Gamma(\tau,y) e^{\partial_t\phi(t_0,x_0)\tau+\partial_x\phi(t_0,x_0)y} dy d\tau,$$
where we use the first two parts of \eqref{eq:h4}. This shows that $w_*$ is a viscosity super-solution of \eqref{whjEq1'}.

Next, we verify that $w^*$ is a viscosity sub-solutions of \eqref{whjEq1'}. We argue that it is enough to verify that $w^*$ is a viscosity subsolution of 
    \begin{equation}\label{whjEq1''}
\min\left\{w,\underline{H}(x/t,\partial_t w,\partial_x w)\right\}=0 \quad \text{ in }(0,\infty)\times (0,\infty),
\end{equation}
which is obtained from \eqref{whjEq1'} by replacing the Hamiltonian $H(s,q,p)$ therein by
$$\underline{H}(s,q,p) = q+p^2+\underline{R}_1(s)+\underline{R}_2(s)\int_0^{\tau_0}\int_\mathbb{R} \Gamma(\tau,y) e^{q\tau + py}\,dyd\tau,$$
where $\underline{R}_i(s)$ are given in \eqref{eq:h4}. 
Indeed, suppose this is the case, then by \eqref{eq:rhooo} we have $w^*(t,x) = t \rho^*(x/t)$ for some u.s.c. function $\rho^*(s)$. By arguing similarly as in Lemma \ref{lem:samehj} (with $\underline{R}_i$ in place of $R_i$) it follows that $\rho^*(s)$ satisfies, in viscosity sense,
\begin{equation}\label{rhohjEq1''}
\min\left\{\rho,\underline H(s,\rho - s\rho',\rho')\right\}\leq 0 \quad \text{ in }(0,\infty).
\end{equation}
Since the Hamiltonian in \eqref{rhohjEq1''} is convex in $\rho'$, a direct application of \cite[Proposition 1.14]{Ishii2013} (see also \cite[Chap. II, Prop. 4.1]{Bardi1997}) yields that $\rho^* \in \textup{Lip}_{loc}([0,\infty))$. It then follows from Rademacher's theorem that $\rho^*$ is differentiable a.e. in $(0,\infty)$, so that it satisfies \eqref{rhohjEq1''} a.e. in $(0,\infty)$. Since $R_i(s) = \underline{R}_i(s)$ a.e. (by (H4)), the following differential inequality holds a.e. in $(0,\infty)$
\begin{equation}\label{rhohjEq1'''}
\min\left\{\rho^*,H(s,\rho^* - s(\rho^*)',(\rho^*)')\right\}\leq 0. 
\end{equation}
However, by the convexity of the Hamiltonian, we can again apply \cite[Chap. I, Prop. 5.1]{Bardi1997} to conclude that it in fact satisfies \eqref{rhohjEq1'''} in $(0,\infty)$ in viscosity sense, i.e., $\rho^*$ is a viscosity sub-solution of  \eqref{rhoeq1'}. By Lemma \ref{lem:samehj} with $c_b=\infty$, we see that $w^*(t,x)=t\rho^*(x/t)$ is a viscosity sub-solution of \eqref{whjEq1'}.

Therefore, it remains to show that $w^*$ is a viscosity sub-solution of \eqref{whjEq1''}. Fix a smooth test function $\phi$ and assume 
$w^*-\phi$ has a strict global maximum at some point $(t_0,x_0)\in(0,\infty)\times(0,\infty)$ and $w^*(t_0,x_0)>0$.
We claim that $\underline{H}(x_0/t_0, \partial_t\phi(t_0,x_0),\partial_x\phi(t_0,x_0)) \leq 0$. By the definition of $w^*$ in \eqref{eq:wstar}, there exist a sequence $\epsilon_n\to 0$ and a sequence of points 
$(t_n,x_n) \in (0,\infty)\times(0,\infty)$ 
such that $w^{\ep_n}(t,x)-\phi$ has a global maximum at $(t_n,x_n)$, and satisfy
\begin{equation}\label{qqq:2.12}
    w^{\ep_n}(t_n,x_n) \to w^*(t_0,x_0)>0\quad \text{ and }\quad (t_n,x_n)\to (t_0,x_0)\quad \text{  as }\quad n\to\infty.
\end{equation}
Next, we claim that 
\begin{equation}\label{qq:2.13}
\liminf_{n \to \infty} \frac{f_{2, \ep_n}(t_n - \ep_n \tau, x_n - \ep_n y, u^{\ep_n}( t_n - \ep_n \tau, x_n - \ep_n y)}{u^{\ep_n}( t_n, x_n)}  \geq \underline{R}_2(x_0/t_0) e^{\partial_t \phi(t_0,x_0) \tau + \partial_x \phi(t_0,x_0) y}
\end{equation}
for a.e. $(\tau,y) \in [0,\tau_0]\times \mathbb{R}$. 

For any given $(\tau,y)\in[0,\tau_0]\times \R$, by passing to a subsequence, we may divide into two cases:
 $$
 {\rm(i)} \quad u^{\ep_n}( t_n - \ep_n \tau, x_n - \ep_n y) \to 0, \text{ as }n\to\infty,\quad \text{ or }\quad  {\rm(ii)} \quad \inf_{n} u^{\ep_n}( t_n - \ep_n \tau, x_n - \ep_n y) > 0.
 $$
In case (i), we use (H1) to obtain
 \begin{align*}
& \frac{f_{2, \ep_n}(t_n - \ep_n \tau, x_n - \ep_n y, u^{\ep_n}( t_n - \ep_n \tau, x_n - \ep_n y)}{u^{\ep_n}( t_n, x_n)} \\
&\geq  (\partial_v f_{2, \ep_n}(t_n - \ep_n \tau, x_n - \ep_n y, 0)-\eta') \exp\left(\frac{w^{\ep_n}(t_n,x_n) - w^{\ep_n}(t_n - \ep_n \tau, x_n - \ep_n y}{\ep_n}\right) \\
     &\geq (\partial_v f_{2, \ep_n}(t_n - \ep_n \tau, x_n - \ep_n y, 0)-\eta') \exp\left(\frac{\phi(t_n,x_n) - \phi(t_n - \ep_n \tau, x_n - \ep_n y}{\ep_n}\right) ,
 \end{align*}
 where the first inequality follows from (H1), and then the second inequality from $(w^{\ep_n}-\phi)(t_n,x_n) = \max (w^{\ep_n} - \phi)$. Letting $n \to \infty$, we deduce 
$$
\liminf_{n \to \infty} \frac{f_{2, \ep_n}(t_n - \ep_n \tau, x_n - \ep_n y, u^{\ep_n}( t_n - \ep_n \tau, x_n - \ep_n y)}{u^{\ep_n}( t_n, x_n)}  \geq (\underline{R}_2(x_0/t_0) - \eta')e^{\partial_t \phi(t_0,x_0) \tau + \partial_x \phi(t_0,x_0) y}.
$$
Since $\eta'>0$ is arbitrarily small, we obtain \eqref{qq:2.13}. 

In case (ii), observe that $f_{2, \ep_n}(t_n - \ep_n \tau, x_n - \ep_n y, u^{\ep_n}( t_n - \ep_n \tau, x_n - \ep_n y)$ is bounded from below by a positive number, and that $u^{\ep_n}(t_n,x_n) = \exp(-w^{\ep_n}(t_n,x_n)/\ep_n) \to 0^+$ (using \eqref{qqq:2.12} and that $w^*(t_0,x_0) >0$), and \eqref{qq:2.13} automatically holds. 

Having proved \eqref{qq:2.13}, we conclude by Fatou's lemma that 
\begin{align}\label{qq:2.14}
\liminf_{n\to\infty}&\int_0^{\tau_0}\int_\mathbb{R}  \Gamma(\tau,y)\frac{f_{2, \ep_n}(t_n - \ep_n \tau, x_n - \ep_n y, u^{\ep_n}( t_n - \ep_n \tau, x_n - \ep_n y)}{u^{\ep_n}( t_n, x_n)} \,dyd\tau \\
&\geq \underline{R}_2(x_0/t_0) \int_0^{\tau_0}\int_\mathbb{R}  \Gamma(\tau,y)e^{\partial_t \phi(t_0,x_0) \tau + \partial_x \phi(t_0,x_0) y}\,dyd\tau .  \notag
\end{align}
Moreover, since $u^{\ep_n}(t_n,x_n)\to0$ as $n\to\infty$, we use (H2) again to get for any $\eta'>0$, there exists $n_0$ such that  
\begin{align*}
\frac{ f_{1, \ep_n}(t_n,x_n,u^{\ep_n})}{u^{\ep_n}(t_n,x_n)}\geq \partial_u f_{1, \ep_n}(t_n, x_n, 0)-\eta', \text{ for all }\, n\ge n_0.
\end{align*}
Letting $n\to\infty$ and then $\eta'\to0$, we use (H4) to get
\begin{equation}\label{qq:2.15}
    \liminf_{n\to\infty}\frac{ f_{1, \ep_n}(t_n,x_n,u^{\ep_n})}{u^{\ep_n}(t_n,x_n)}\geq \underline{R}_1(x_0/t_0).
\end{equation}
Now we are ready to verify $\underline{H}(x_0/t_0, \partial_t\phi(t_0,x_0),\partial_x\phi(t_0,x_0)) \leq 0.$ 
Indeed, at the point $(t_n,x_n)$,
\begin{align*}
 &\partial_t\phi-\epsilon_n\partial_{xx}\phi+|\partial_x \phi|^2  +  \frac{f_{1, \ep_n}(t_n,x_n,u^{\ep_n})}{u^{\ep_n}}\\
&\le-\int_0^{\tau_0}\int_\mathbb{R}  \Gamma(\tau,y)\frac{f_{2, \ep_n}(t_n - \ep_n \tau, x_n - \ep_n y, u^{\ep_n}( t_n - \ep_n \tau, x_n - \ep_n y)}{u^{\ep_n}( t_n, x_n)} \,dyd\tau.
\end{align*}
Letting $n\to\infty$, while using  \eqref{qq:2.14} and \eqref{qq:2.15}, we obtain
$$
\partial_t\phi+|\partial_x \phi|^2 + \underline{R}_1(x_0/t_0) \leq - \underline{R}_2(x_0/t_0) \int_0^{\tau_0}\int_\mathbb{R}  \Gamma(\tau,y)e^{\partial_t \phi(t_0,x_0) \tau + \partial_x \phi(t_0,x_0) y}\,dyd\tau.$$ This concludes the proof.
\end{proof}


\begin{corollary}\label{lem:2.2} Let $w^*$ and $w_*$ be given by \eqref{eq:wstar}. 
There exists a u.s.c. function $\rho^*(s)$ $($resp. l.s.c. function $\rho_*(s))$ such that 
\begin{equation}\label{eq:222}
w^*(t,x) = t \rho^*(x/t),\quad (\textit{resp.} \quad  w_*(t,x) = t \rho_*(x/t))\quad  \text{ in }\quad  (0,\infty)\times[0,\infty).
\end{equation}
Moreover, $\rho^*$ $($resp. $\rho_*)$ is a viscosity supersolution $($resp. subsolution$)$ of \eqref{rhoeq1'}.
\end{corollary}
\begin{proof}
By Proposition \ref{prop:2.3}, $w^*$ and $w_*$ are respectively sub- and super-solution of \eqref{whjEq1'}. 
The existence of $\rho^*(s)$ and $\rho_*(s)$ and \eqref{eq:222} are proved in \eqref{eq:rhooo2} and \eqref{eq:rhooo}. The rest follows from Lemma \ref{lem:samehj}.
\end{proof}

The following lemma indicates that each sub-solution $\underline{\rho}(s)$ is strictly increasing 
in the interval $\{s\in(0,\infty):\, \underline\rho(s) >0\}$.

\begin{lemma}\label{lem:2.3}
{Let $\underline{\rho}$ be a {nonnegative} viscosity subsolution to \eqref{rhoeq1'}, such that $\underline\rho(0) = 0$ and $\underline\rho(s) \to \infty$ as $s \to \infty$, then (i) $s \mapsto \underline\rho(s)$ is non-decreasing; (ii) $s \mapsto \frac{\underline\rho(s)}{s}$ has no positive local maximum points in $(0,\infty)$; (iii) $\lim\limits_{s \to \infty} \frac{\underline\rho(s)}{s}$ exists in $[0,\infty]$; (iv) if $\lim\limits_{s \to \infty} \frac{\underline\rho(s)}{s}<\infty$, then $\sup_{s>0} \frac{\underline\rho(s)}{s}<\infty$.}
\end{lemma}
\begin{proof}
By $R_2(s) \geq 0$ and $\underline\rho \geq 0$, it follows that $\underline\rho$ satisfies $-s\underline\rho' + |\underline\rho'|^2 + R_1(s) \leq 0$ and hence
$$
|\underline\rho'|^2 \leq {s^2 + 2\|R_1\|_\infty} \quad \text{ in }(0,\infty)
$$
in the viscosity sense. It follows from \cite[Proposition 1.14]{Ishii2013} that 
$\underline\rho \in \textup{Lip}_{loc}([0,\infty))$. We first show assertion (i). 
Since $\underline\rho(0) = 0$ and $\underline\rho(s) \to \infty$ as $s \to \infty$, it suffices to show that there does not exist $s_0>0$ such that $\underline\rho(s_0)>0$ and $s\mapsto \underline\rho(s)$ has a local maximum points at $s_0$. Assume to the contrary, then by definition of viscosity solution (using $\phi\equiv 0$ as test function) we have
$$
R_1(s_0)+R_2(s_0)<\underline\rho(s_0) + R_1(s_0) + R_2(s_0) \int_0^{\tau_0} \int_\mathbb{R} \Gamma(\tau,y) e^{\tau \underline\rho(s_0)}\,d\tau dy \leq 0,
$$
which is a contradiction to (H4).  This proves assertion (i).

The assertion (ii) can be proved similarly by considering
 $\sigma(s) = \rho(s)/s$, which is a viscosity subsolution of
$$
\min\{\sigma,-s^2 \sigma' + |\sigma + s\sigma'|^2 + R_1(s) + R_2(s) \int_0^{\tau_0} \int_\mathbb{R} \Gamma(\tau,y) e^{ -s^2 \sigma' \tau + (\sigma + s\sigma')y}\,d\tau dy \} = 0.
$$
Next, we show (iii).
Suppose there exists $s_0 >0$ such that 
$\sigma$ attains a local maximum at some $s_0 >0$ and that $\sigma(s_0) >0$. This implies
$$
R_1(s_0)+R_2(s_0)<|\sigma(s_0)|^2 + R_1(s_0) + R_2(s_0) \int_0^{\tau_0} \int_\mathbb{R} \Gamma(\tau,y) e^{\sigma(s_0)y}\,d\tau dy \leq 0,
$$
which is a contradiction. This proves (ii). (iii) follows directly from (ii).

Next, we assume $\lim\limits_{s \to \infty} \frac{\underline\rho(s)}{s}<\infty$ and prove (iv).
Recall that $\underline\rho \in \textup{Lip}_{loc}([0,\infty))$ and 
$\underline\rho(0) = 0$, so that $\lim\limits_{s \to 0^+} \underline\rho(s)/s <\infty$. By (ii), we  have $\sup\limits_{s >0} {\underline\rho(s)}/{s} \leq \max\left\{\lim\limits_{s \to 0^+} \underline\rho(s)/s, \lim\limits_{s \to \infty} \underline\rho(s)/s \right\} < \infty$.
\end{proof}

    By Lemma \ref{lem:2.3}(iii), $\lim\limits_{s \to \infty} \frac{\underline{\rho}(s)}{s} $ exists.
By dividing in to the following two cases,
$$
{\rm(i)}\,\, 0<\lim_{s \to \infty} \frac{\underline{\rho}(s)}{s} < \infty,\quad {\rm(i')}\,\, \lim_{s \to \infty} \frac{\underline{\rho}(s)}{s}=\infty, 
$$
we will give several novel conditions so that \eqref{rhoeq1'} admits a comparison principle. 
\begin{proposition}\label{cp:rho1}
Suppose $R_i(s),i=1,2$ be given by {\rm(H4)}.
    Let $\overline{\rho}$ and $\underline\rho$ be non-negative super- and sub-solutions of \eqref{rhoeq1'} in $(0,\infty)$, such that 
\begin{equation}
\underline{\rho}(0) \leq  \overline\rho(0),\quad \text{ and }\quad \lim_{s \to \infty} \frac{\underline{\rho}(s)}{s} \leq \liminf_{s \to \infty}  \frac{\overline{\rho}(s)}{s}.
\end{equation}
If $\lim\limits_{s \to \infty} {\underline\rho(s)}= \infty$ and $\lim\limits_{s \to \infty} \frac{\underline\rho(s)}{s} <\infty$, then $\underline\rho(s) \leq \overline\rho(s)$ in $[0,\infty)$.
\end{proposition}
\begin{proof}
See Appendix \ref{sec:A}.
\end{proof}
\begin{proposition}\label{cp:rho}
Let $R_i(s),i=1,2$ be given by {\rm(H4)}.
    Let $\overline{\rho}$ and $\underline\rho$ be non-negative super- and sub-solutions of \eqref{rhoeq1'} in $(0,\infty)$, such that 
\begin{equation}
\underline{\rho}(0) \leq  \overline\rho(0),\quad \text{ and }\quad \lim_{s \to \infty} \frac{\underline{\rho}(s)}{s} \leq \liminf_{s \to \infty}  \frac{\overline{\rho}(s)}{s}.
\end{equation}
If $\lim\limits_{s \to \infty} \frac{\underline\rho(s)}{s} =\infty$ and one of the following holds:
\begin{itemize}
    \item [{\rm(i)}] $s\mapsto R_2(s)$ is non-increasing in $[s_0,\infty)$ for some $s_0$,
    \item [{\rm(ii)}] $\lim\limits_{s \to +\infty}R_2(s)$ exists and is positive, and 
 $\displaystyle R_2(s) \geq  \left(\sup_{(s,\infty)}R_2\right) - o\left(\frac{1}{s^2}\right)$ for $s \gg 1$,
     


\item [{\rm(iii)}] One of $\lim\limits_{s \to +\infty}R_i(s)$ $(i=1,2)$ exists, and $\Gamma(\tau,y) = 0$ in $[0,\tau_1]\times (-\infty,0),$ for some $\tau_1 \in (0,\tau_0]$, and $\limsup\limits_{s \to +\infty} R_2(s) >0$,

\end{itemize}
then  $\underline{\rho}(s) \leq \overline\rho(s)$ in $[0,\infty)$.
\end{proposition}

\begin{proof}
We first prove case (ii), recall from the proof of Lemma \ref{lem:2.3} that $\underline{\rho}$ is locally Lipchitz continuous, so that it is differentiable in $\mathcal{S}$, where $[0,\infty) \setminus \mathcal{S}$ has zero measure. Since $\underline{\rho}$ is a viscosity subsolution, it must satisfy 
\begin{equation}\label{irhoeq1'}
  \min \{\underline{\rho}, \underline{\rho} - s \underline{\rho}' + |\underline{\rho}'| + \int_0^{\tau_0} \hspace{-.1cm} \int_\mathbb{R} \Gamma(\tau,y) e^{\tau (\underline{\rho}-s\underline{\rho}') + y \underline{\rho}'}\,dyd\tau \}\leq 0 \quad \text{ for each }s \in \mathcal{S}.
\end{equation}
Since $\underline\rho(\infty) = \infty$ and that $R_1$ is upper semi-continuous and locally monotone (by (H4)) and that $R_2(+\infty) >0$, we can choose a sequence $\{s_k\} \subset \mathcal{S}$ such that 
\begin{equation}\label{eq:sk}
s_k \to \infty, \quad \underline\rho(s_k) >1,\quad R_1(s_k) \to \limsup_{s \to \infty} R_1(s) \quad  \text{ and }\quad  \inf\limits_{k} R_2(s_k) >0.
\end{equation}
For each $k$, denote $a_k=\underline\rho(s_k)$, $b_k=\underline\rho'(s_k)$, then specializing \eqref{irhoeq1'} at $s=s_k$ gives
\begin{equation}\label{eq:2.23b}
a_k - s_k b_k + |b_k|^2 + R_1(s_k) +  R_2(s_k)\int_0^{\tau_0}\int_\mathbb{R} \Gamma(\tau,y) e^{(a_k - s_k b_k)\tau + b_k y}\,dyd\tau \leq 0.
\end{equation}
Using $a_k \geq 0$ and $R(s_k)>0$ we have $|b_k|^2 \leq s_k b_k + |R_1|_\infty$ and hence $|b_k| \leq O(s_k)$. Using the latter, along with $\inf\limits_k R_2(s_k)>0$ in \eqref{eq:2.23b}, we deduce
\begin{equation}\label{eq:2.23c}
\int_0^{\tau_0}\int_\mathbb{R} \Gamma(\tau,y) e^{(a_k - s_k b_k)\tau + b_k y}\,dyd\tau \leq O(|s_k|^2).\end{equation}

Next, define
$$
\nu_k = \left[\sup_{[s_k,\infty)}R_1 - R_1(s_k)\right] +\max\left\{0,\sup_{[s_k,\infty)}R_2 - R_2(s_k)\right\} \int_0^{\tau_0}\int_\mathbb{R} \Gamma(\tau,y) e^{(a_k - s_k b_k)\tau + b_k y}\,dyd\tau,
$$
then we have $\nu_k \to 0$. Indeed, $\left[\sup_{[s_k,\infty)}R_1 - R_1(s_k)\right] \to 0$ by our choice of $s_k$, and 
\begin{align*}
\left[\sup_{[s_k,\infty)}R_2 - R_2(s_k)\right]\int_0^{\tau_0}\int_\mathbb{R} \Gamma(\tau,y) e^{(a_k - s_k b_k)\tau + b_k y}\,dyd\tau&\leq o\left(\frac{1}{|s_k|^2}\right)  O(|s_k|^2) = o(1),
\end{align*}
where we used  the assumption $\displaystyle R_2(s_k) \geq  \left(\sup_{(s_k,\infty)}R_2\right) - o\left(\frac{1}{s_k^2}\right)$ and \eqref{eq:2.23c}.

Define $\underline\rho_k \in \textup{Lip}_{loc}([0,\infty))$ by 
$$
\underline\rho_k(s):= \begin{cases}
\underline\rho(s) -\nu_k  &\text{ for }s \in [0,s_k),\\
\underline\rho(s_k)  + (s-s_k) \underline\rho'(s_k)  - \nu_k&\text{ for }s \in [s_k,\infty).
\end{cases}
$$
We claim that 
$\underline{\rho}_k(s)$ satisfies 
\begin{equation}\label{rholeq}
    \underline{\rho}_k(s)-s\underline{\rho}_k'(s)+|\underline{\rho}_k'(s)|^2+{R}_1(s)+{R}_2(s)\int_{0}^{\tau_0}\int_{\R }\Gamma(\tau, y)e^{(\underline{\rho}_k(s)-s\underline{\rho}_k'(s))\tau+\underline{\rho}_k' y} dy d\tau \leq 0 
\end{equation}
in the viscosity sense in $(0,\infty)$. 
Indeed, it is easy to see that $\underline{\rho}_k(s)$ remains a viscosity subsolution to \eqref{rholeq} in $[0,s_k)$. It remains to show that it is a classical solution to \eqref{rholeq} in $[s_k,\infty)$. Indeed, if we denote $a_k=\underline\rho(s_k)$ and $b_k=\underline\rho'(s_k)$), then
for $s \geq s_k$, 
\begin{align*}
    &\underline{\rho}_k(s)-s\underline{\rho}_k'(s)+|\underline{\rho}_k'(s)|^2+{R}_1(s)+{R}_2(s)\int_{0}^{\tau_0}\int_{\R }\Gamma(\tau, y)e^{(\underline{\rho}_k(s)-s\underline{\rho}_k'(s))\tau+\rho' y} dy d\tau\\
    &\leq a_k - s_k b_k + |b_k|^2 - \nu_k + \left(\sup_{[s_k,\infty)}R_1\right) +\left( \sup_{[s_k,\infty)}R_2 \right)\int_{0}^{\tau_0}\int_{\R }\Gamma(\tau,y)e^{(a_k - s_k b_k)\tau+b_k y} dy d\tau\\
    &\leq a_k - s_k b_k + |b_k|^2 + R_1(s_k) + R_2(s_k) \int_{0}^{\tau_0}\int_{\R }\Gamma(\tau,y)e^{(a_k - s_k b_k)\tau+b_k y} dy \leq 0.
\end{align*}
where we used 
$$
\underline{\rho}_k(s) - s\underline\rho'_k(s) = \underline{\rho}(s_k) - s_k\underline\rho'(s_k) - \nu_k = a_k - s_k b_k - \nu_k \quad \text{ for }s \geq s_k
$$
for the first inequality, the definition of $\nu_k$ for the second inequality, 
and \eqref{eq:2.23b} for the last inequality. This proves that 
$\underline{\rho}_k$ is a viscosity subsolution of \eqref{rhoeq1'} in $(0,\infty)$.

Now, note that $\overline\rho$ and $\underline\rho_k$ form a pair of viscosity super- and subsolution of \eqref{rhoeq1'} that satisfies the setting of Proposition \ref{cp:rho1}. By comparison, it follows in particular that
$$\underline\rho(s)- \nu_k  \leq \overline\rho(s)  \quad \text{ for }s \in [0,s_k].$$
Letting $k \to \infty$, then $s_k \to \infty$ and $\nu_k \to 0$, and the desired conclusion follows. This proves case (ii). Case (i) can be proven exactly as case (ii) (but the assumption $R_2(+\infty)>0$ is not needed).

Next, we show (iii). In this case, we choose $s_k \in \mathcal{S}$ such that
\begin{equation}\label{eq:2.23d}
R_1(s_k) \to \limsup\limits_{s \to \infty} R_1(s_k) \quad \text{ and } \quad R_2(s_k) \to \limsup\limits_{s \to \infty} R_2(s_k).
\end{equation}
This is possible in view of the assumption of case (iii), and local monotonicity of $s\mapsto R_i(s)$. 
Denote $a_k = \underline\rho(s_k)$, 
and $b_k = \underline\rho'(s_k)$. From \eqref{eq:2.23b}, 
we observe that
\begin{equation}\label{eq:rhoee1}
a_k  - s_k b_k \leq \|R_1\|_{\infty} \quad \text{ for all }k \geq 1.
\end{equation}

Fix an arbitrary $\nu>0$, and choose $\tau' \in (0,\tau_1)$ and $k'\in \mathbb{N}$ such that for all $k \geq k'$,
$$
\max\left\{\sup_{[s_k,\infty)} R_1 - R_1(s_k), \|R_2\|_\infty  \int_0^{\tau'}  \int_\mathbb{R} \Gamma(\tau,y)e^{\tau\|R_1\|_\infty}\,dyd\tau \right\}< \frac{\nu}{2},
$$
which is possible in view of \eqref{eq:2.23d}. 
Define function $\underline{\rho}_k(s)$ by 
$$
\underline\rho_k(s):= \begin{cases}
\underline\rho(s) -\nu  &\text{ for }s \in [0,s_k],\\
\underline\rho(s_k)  + (s-s_k) \underline\rho'(s_k)  - \nu&\text{ for }s \in (s_k,\infty).
\end{cases}
$$
We claim that $\underline{\rho}_k$ is a viscosity subsolution to \eqref{rhoeq1'} for $k \gg 1$. Again, it suffices to show, for $k\gg 1$, that $\underline{\rho}_k$ is a classical solution of \eqref{rholeq}  
 in $[s_k,\infty)$.

To this end, observe 
the assumption of case (iii) implies for each $\tau' \in (0,\tau_1]$,
\begin{align}
&\int_0^{\tau_0} \int_\mathbb{R} \Gamma(\tau,y) e^{(a_k - sb_k)\tau - b_k y}\,dyd\tau \notag \\
&=  \int_0^{\tau'}  \int_0^{\infty} \Gamma(\tau,y) e^{(a_k - sb_k)\tau - b_ky}\,dyd\tau + \int_{\tau'}^{\tau_0} \int_\mathbb{R} \Gamma(\tau,y) e^{(a_k - sb_k)\tau - b_k y}\,dyd\tau \notag \\
&\leq \int_0^{\tau'}  \int_0^{\infty} \Gamma(\tau,y) e^{\tau \|R_1\|_\infty}\,dyd\tau+ \int_{\tau'}^{\tau_0} \int_\mathbb{R} \Gamma(\tau,y) e^{(a_k - sb_k)\tau - b_k y}\,dyd\tau, \label{eq:rhoeee}
\end{align}
where $a_k = \underline\rho(s_k)$, and $b_k = \underline\rho'(s_k)$, and we used \eqref{eq:rhoee1} and $b_k \geq 0$ (as $\underline\rho'(s) \geq 0$ for all $s$).

Observe that for $k \gg 1$ and $s \in [s_k,\infty)$, 
\begin{align*}
&\underline\rho_k(s) - s \underline\rho_k'(s) + |\underline\rho_k'(s)|^2 + R_1(s) + R_2(s) \int_0^{\tau_0} \int_\mathbb{R} \Gamma(\tau,y) e^{(\underline\rho_k(s) - s\underline\rho_k'(s))\tau - \underline\rho_k'(s) y}\,dyd\tau \notag\\
  &=   a_k - \nu - s_k b_k + |b_k|^2 + R_1(s) + R_2(s) \int_0^{\tau_0} \int_\mathbb{R} \Gamma(\tau,y) e^{(a_k - \nu- s_kb_k)\tau - b_k y}\,dyd\tau  \\
  &\leq a_k  - s_k b_k + |b_k|^2 + R_1(s_k) + R_2(s) \int_{\tau'}^{\tau_0} \int_\mathbb{R} \Gamma(\tau,y) e^{(a_k - \nu- s_kb_k)\tau - b_k y}\,dyd\tau  \\
  &\leq a_k  - s_k b_k + |b_k|^2 + R_1(s_k) + e^{-\nu \tau'}(\sup\limits_{[s_k,\infty)}R_2(s)) \int_{\tau'}^{\tau_0} \int_\mathbb{R} \Gamma(\tau,y) e^{(a_k - s_kb_k)\tau - b_k y}\,dyd\tau \\
  &\leq a_k  - s_k b_k + |b_k|^2 + R_1(s_k) + R_2(s_k)\int_{\tau'}^{\tau_0} \int_\mathbb{R} \Gamma(\tau,y) e^{(a_k - s_kb_k)\tau - b_k y}\,dyd\tau \leq 0,
\end{align*}
where we used our choice of $\nu$, $\tau'$ and \eqref{eq:rhoeee} for the first inequality, and that 
$\frac{ \sup_{[s_k,\infty)} R_2 }{R_2(s_k)} \to 1$ (since $R_2(s_k) \to \limsup\limits_{s \to \infty} R_2>0$) in the third inequality.
Hence, \eqref{rholeq} holds for $k \gg 1$ and $s \in [s_k,\infty)$. 
Now, $\overline\rho$ and $\underline\rho_k$ defines a pair of viscosity super- and sub-solution of \eqref{rhoeq1'} that satisfies the setting of Proposition \ref{cp:rho1}. Comparison implies that, for each $\nu>0$, 
$$
\underline\rho(s) - \nu \leq \overline\rho(s) \text{ for }s \in [0,s_k] \text{ and }k \gg 1.
$$
Letting $k \to \infty$ and $\nu \searrow 0$, we get $\underline\rho(s) \leq \overline\rho(s)$ in $[0,\infty)$. 
\end{proof}

\begin{proof}[Proof of Propositions \ref{prop:1.2}] 
Suppose (H1)-(H5) hold, and either (i) $\mu \in (0,\infty)$ or (ii) $\mu = \infty$ and one of (H6), (H6$'$) or (H6$''$) holds.
By Proposition \ref{cp:rho}, there is at most one solution to \eqref{rhoeq1''} subject to the boundary conditions $\rho(0)=0$ and $\lim_{s \to \infty} \frac{\rho(s)}{s} = \mu\in(0,\infty]$.

To show existence, let $w_*(t,x)$ and $w^*(t,x)$ be given by \eqref{eq:wstar}, and let 
$\rho_*(s)$ and $\rho^*(s)$ be respectively the super- and sub-solution of \eqref{rhoeq1''} that are given in Corollary \ref{lem:2.2}, i.e. 
$$
w_*(t,x)=t\rho_*\left(\frac{x}{t}\right) \quad \text{ and }\quad w^*(t,x)=t\rho^*\left(\frac{x}{t}\right) \text{ for }(t,x) \in (0,\infty)\times[0,\infty).
$$
By construction in \eqref{eq:wstar}, $w_* \leq w^*$ in $(0,\infty)\times[0,\infty)$, and hence $\rho_* \leq \rho^*$.  

We claim that
\begin{equation}\label{eq:prop:2.1.2}
    \rho_*(0) = \rho^*(0) = 0 \quad \text{ and }\quad \lim_{s \to \infty} \frac{\rho_*(s)}{s} =  \lim_{s \to \infty} \frac{\rho^*(s)}{s}=\mu\in(0,\infty].
\end{equation}
The assertions follow from Proposition \ref{prop:2.1}, since
\begin{equation}\label{eq:prop:2.1.2a}
\rho_*(0) = w_*(1,0) = 0 \quad \text{ and }\quad \rho^*(0) = w^*(1,0) = 0,
\end{equation}
where we used the first part of \eqref{eq:prop:2.1.1}. Also,
\begin{equation}\label{eq:prop:2.1.2b}
\liminf_{s \to \infty} \frac{\rho_*(s)}{s} = \liminf_{s \to \infty} w_*\left(\frac{1}{s},1\right) \geq  w_*(0,1) =  \mu\in(0,\infty], 
\end{equation}
where we used the second part of \eqref{eq:prop:2.1.1}, and that $w_*$ is l.s.c.. Similarly, we have 
\begin{equation}\label{eq:prop:2.1.2c}
\limsup_{s \to \infty} \frac{\rho^*(s)}{s} = \limsup_{s \to \infty} w^*\left(\frac{1}{s},1\right) \leq  w^*(0,1) =  \mu\in(0,\infty]. 
\end{equation}
We can combine \eqref{eq:prop:2.1.2a}-\eqref{eq:prop:2.1.2c} to obtain \eqref{eq:prop:2.1.2}. This, and the fact that $\rho_*$ and $\rho^*$ are the super- and sub-solution to \eqref{rhoeq1''} enables the application of the comparison result (Proposition \ref{cp:rho}), which implies that $\rho^* \leq \rho_*$ in $(0,\infty)\times [0,\infty)$. Recalling that $\rho^* \geq \rho_*$ in $(0,\infty)\times [0,\infty)$ by construction, we conclude $\rho^* \equiv \rho_*$. This provides the existence of a viscosity solution $\hat\rho$ to \eqref{rhoeq1''} satisfying $\hat\rho(0) = 0$ and $\lim_{s \to \infty} \frac{\hat\rho(s)}{s} = \mu\in(0,\infty]$. 

Finally, Lemma \ref{lem:2.3}(i) says that $\hat\rho(s)$ is non-decreasing, and so $s_{nlp}\in[0,\infty)$ is well-defined. 
\end{proof}

\begin{proof}[Proof of Theorem \ref{thm:1.3}]
Suppose (H1)-(H5) hold, and let $u(t,x)$ be a solution of \eqref{Model} with initial data satisfying (IC$_\mu$) for some $\mu \in (0,\infty)$. Then Proposition \ref{cp:rho} is applicable.

Let $w^*(t,x)$ and $w_*(t,x)$ be given by \eqref{eq:wstar}. From the proof of Proposition \ref{prop:1.2}, we have
$$
w^*(t,x) = t\hat\rho(x/t)  = w_*(t,x) \quad \text{ for }(t,x) \in (0,\infty)\times[0,\infty),
$$
where $\hat\rho(s)$ is the unique viscosity solution to \eqref{rhoeq1''} subject to the boundary conditions $\rho(0)=0$ and $\lim\limits_{s \to \infty} \frac{\rho(s)}{s} = \mu$. Since $\hat\rho(s) = 0$ in $[0,\hat{s}^\mu]$ and $\hat\rho(s) >0$ in $(\hat{s}^\mu,\infty)$, we deduce that
\begin{equation}\label{eq:thm:1.3.1}
w^\ep(t,x) \to 0  \quad \text{ locally uniformly in }\{(t,x): \, 0 \leq x/t < \hat{s}^\mu\}
\end{equation}
and that 
\begin{equation}\label{eq:thm:1.3.2}
\liminf_{\ep \to 0} \inf_{K} w^\ep(t,x) >0 \text{ for each }K \subset\subset \{(t,x): \,  x > \hat{s}^\mu t\}.
\end{equation}

We show the first part of \eqref{eq:1.5}. First, by Remark \ref{rmk:largesp}, there exists some  $\overline s>0$ sufficiently large such that
\begin{equation}
    \lim_{t\to\infty} \sup_{x \geq \overline s t} u(t,x) =0.
\end{equation}
Now for given $\eta>0$, take $K = \{(1,x):\hat{s}^\mu + \eta \leq x \leq \overline{s}\}$ in \eqref{eq:thm:1.3.2}, then 
$$
\lim_{t\to\infty} \sup_{(s_{nlp} + \eta)t \leq x \leq \overline s t} u(t,x) = \lim_{\ep \to 0} \sup_{s_{nlp} + \eta\leq x'\leq \overline s} u\left( \frac{1}{\ep}, \frac{x'}{\ep}  \right) =\lim_{\ep \to 0} \sup_{s_{nlp} + \eta\leq x'\leq \overline s} \exp\left( -\frac{w^\ep\left( 1, x'\right)}{\ep} \right)= 0 . 
$$
This proves the first part of \eqref{eq:1.5}.

To show the second part of \eqref{eq:1.5}, fix $(t_0,x_0)$ such that $x_0/t_0 < \hat{s}^\mu$ and suppose to contrary that there a sequence $\ep = \ep_k \to 0$ and $(t_\ep, x_\ep) \to (t_0,x_0)$ such that
$$
0 < \frac{x_0}{t_0}< \hat{s}^\mu \quad \text{ and }\quad u^\ep(t_\ep,x_\ep)  \to 0.
$$
Now, consider the test function $\phi_\ep(t,x):= |t-t_\ep|^2 + |x-x_\ep|^2$. Since $w^\ep(t,x) \to 0$ uniformly on a neighborhood $B_{2r}(x_0,t_0)$, by taking $\ep$ so small that $(t_\ep,x_\ep) \in B_r(x_0,t_0)$, 
we see that $w^\ep - \phi_\ep$ has an interior maximum point $(t'_\ep,x'_\ep) \in B_{2r}(x_0,t_0)$. 
Observe that 
\begin{equation}\label{eq:2.13}
(t'_\ep,x'_\ep) \to (t_0,x_0) \quad \text{ and }\quad (t_\ep,x_\ep) \to (t_0,x_0) \quad \text{ as }\epsilon \to 0.\end{equation}
And, by construction,
\begin{equation}\label{eq:2.14}
  w^\ep(t'_\ep,x'_\ep)\geq(w^\ep - \phi_\ep)(t'_\ep,x'_\ep)\geq  (w^\ep - \phi_\ep)(t_\ep,x_\ep) = w^\ep(t_\ep,x_\ep) 
\end{equation}
which, in view of $u^\ep(t,x) = \exp(-\frac{1}{\ep} w^\ep(t,x))$, implies that $0 \leq u^\ep(t'_\ep,x'_\ep) \leq   u^\ep (t_\ep,x_\ep)$. Since $u^\ep (t_\ep,x_\ep) \to 0$ by assumption, we also have
\begin{equation}\label{eq:2.15}
 u^\ep(t'_\ep,x'_\ep) \to 0 \quad \text{ as }\epsilon \to 0.
\end{equation}

Next, fix $0<\eta'<1$ such that 
\begin{equation}\label{eq:2.16}
\left[-\eta' + \underline{R}_2\left(\frac{x_0}{t_0}\right)\right](1-\eta')+ \underline{R}_1\left(\frac{x_0}{t_0}\right) >2\eta'.
\end{equation}
Note that the above holds when $\eta'=0$ (by \eqref{eq:H7}), so it also holds for small $\eta'>0$ by continuity. 

Next, fix $M>1$ such that 
\begin{equation}\label{eq:2.17}
\int_0^{\tau_0} \int_{-M}^M \Gamma(\tau,y)\,dyd\tau \geq (1-\eta').
\end{equation}
Then by the definition of $\phi_\ep$ and \eqref{eq:2.13}, and passing to a subsequence if necessary, we see that
\begin{equation}\label{eq:2.18}
\begin{cases}
\partial_t \phi_\ep(t'_\ep,x'_\ep) \to 0,\quad \partial_x \phi_\ep(t'_\ep,x'_\ep) \to 0, \quad \sup\limits_{B_{2r}}|D^2 \phi_\ep| \leq 2,\\
 \int_0^{\tau_0} \int_{-M}^M \Gamma(\tau,y) \partial_u f_{2,\ep}(t'_\ep - \ep \tau, x'_\ep - \ep y,0) e^{\frac{\phi_\ep(t'_\ep,x'_\ep) - \phi_\ep(t'_\ep - \ep \tau, x'_\ep - \ep y)}{\ep}}\,dyd\tau\\\quad \geq  \underline{R}_2\left(\frac{x_0}{t_0}\right) \int_0^{\tau_0} \int_{-M}^M \Gamma(\tau,y)\,dyd\tau + o(1),
\end{cases}
\end{equation}
where $B_{2r}=B_{2r}(t'_\ep, x'_\ep)$ in the supremum, and 
$\underline{R}_2\left(\frac{x_0}{t_0}\right)$ is given in \eqref{eq:h4} and we used \eqref{eq:H7}.

Having chosen $M$, we claim that \eqref{eq:2.15} can be strengthened to
\begin{equation}\label{eq:2.15a}
\sup\limits_{0 \leq \tau \leq \tau_0 \atop |y| \leq M} u^\ep(t'_\ep - \ep \tau, x'_\ep - \ep y)=\sup\limits_{0 \leq \tau \leq \tau_0 \atop |y| \leq M}u\left(\frac{t'_\ep}{\ep} -  \tau, \frac{x'_\ep}{\ep} -  y\right) \to 0 \quad \text{ as }\ep \to 0.
\end{equation}
Indeed, we can rewrite \eqref{Model} as
\begin{equation}\label{eq:2.15b}
\partial_t u -  \partial_{xx} u \geq  -C_0 u \quad \text{ in }(0,\infty)\times\mathbb{R}, 
\end{equation}
since $0 \leq u^\ep \leq M_0$ for some $M_0>0$. Passing to a sequence, we may assume that $$\tilde{u}_\ep(t,x) = u\left(\frac{t'_\ep}{\ep} + t, \frac{x'_\ep}{\ep} + x \right) \to \tilde{u}_0(t,x) \quad \text{ in }C_{loc}(\mathbb{R}\times\mathbb{R}).$$
Moreover, the limit function $\tilde{u}_0$ is a non-negative weak solution of \eqref{eq:2.15b} such that $\tilde{u}_0(0,0)=0$. By the strong maximum principle, we deduce that $\tilde{u}_0(t,x) \equiv 0$ for $t \leq 0$ and $x \in \mathbb{R}$. This shows that $\tilde{u}_\ep \to 0$ locally uniformly in $(-\infty,0]\times\mathbb{R}$, which implies \eqref{eq:2.15a}.

In view of \eqref{eq:H2} and \eqref{eq:2.15a}, we may consider $\ep$ small enough such that
\begin{align}
 \int_0^{\tau_0}\int_{-M}^M \Gamma(\tau,y) \left[ -\eta' + \partial_u f_{2,\ep}(t'_\ep - \ep\tau, x'_\ep - \ep y,0)\right]u^\ep(t'_\ep - \ep \tau, x'_\ep - \ep y)\,dyd\tau \notag\\
\leq \int_0^{\tau_0}\int_{-M}^M \Gamma(\tau,y)  f_{2,\ep} (t'_\ep - \ep\tau, x'_\ep - \ep y, u^\ep(t'_\ep - \ep \tau, x'_\ep - \ep y))\,dyd\tau,
\label{eq:2.19}
\end{align}
and similarly,
\begin{equation}\label{eq:2.20}
(1-\eta') \partial_u  f_{1, \ep}(t'_\ep,x'_\ep,0)u(t'_\ep,x'_\ep) \leq f_{1, \ep}(t'_\ep,x'_\ep,u(t'_\ep,x'_\ep)) .
\end{equation}
Hence, (below $w^\ep,\phi_\ep$ and their derivatives are evaluated at $(t'_\ep,x'_\ep)$, unless otherwise stated)
 \begin{align*}
    &\left[-\eta' + \underline{R}_2\left(\frac{x_0}{t_0}\right)\right](1-\eta')+ o(1)\leq \left[-\eta'+\underline{R}_2\left(\frac{x_0}{t_0}\right)\right] \int_0^{\tau_0} \hspace{-.2cm}\int_{-M}^M \Gamma(\tau,y)\,dyd\tau + o(1)\\
    &= \partial_t \phi_\ep \hspace{-.1cm} -\ep \partial_{x}^2 \phi_\ep \hspace{-.1cm} - |\partial_x \phi_\ep|^2 \hspace{-.1cm} + \hspace{-.1cm} \int_0^{\tau_0} \hspace{-.2cm} \int_{-M}^M \Gamma(\tau,y) [\partial_u f_{2,\ep}(t'_\ep\! - \!\ep \tau, x'_\ep - \ep y,0)-\eta'] e^{\frac{\phi_\ep(t'_\ep,x'_\ep) - \phi_\ep(t'_\ep - \ep \tau, x'_\ep - \ep y)}{\ep}}\,dyd\tau \\
    &\leq \partial_t w^\ep \hspace{-.1cm} -\ep \partial_{x}^2 w^\ep \hspace{-.1cm} - |\partial_x w^\ep|^2  \hspace{-.1cm} +  \hspace{-.1cm} \int_0^{\tau_0}  \hspace{-.2cm} \hspace{-.1cm} \int_{-M}^M  \hspace{-.1cm} \Gamma(\tau,y)[\partial_u f_{2,\ep}(t'_\ep - \ep \tau, x'_\ep - \ep y,0)-\eta'] e^{\frac{w^\ep(t'_\ep,x'_\ep) - w^\ep(t'_\ep - \ep \tau, x'_\ep - \ep y)}{\ep}}\,dyd\tau \\
    &= \partial_t w^\ep \hspace{-.1cm} -\ep \partial_{x}^2 w^\ep - |\partial_x w^\ep|^2 \hspace{-.1cm} + \hspace{-.1cm} \int_0^{\tau_0}\hspace{-.2cm} \int_{-M}^M \Gamma(\tau,y)  \frac{[\partial_u f_{2,\ep}(t'_\ep - \ep \tau, x'_\ep - \ep y,0)-\eta']u^\ep(t'_\ep - \ep \tau, x'_\ep - \ep y)}{u^\ep(t'_\ep,x'_\ep)}\,dyd\tau \\
    &\leq \partial_t w^\ep -\ep \partial_{x}^2 w^\ep - |\partial_x w^\ep|^2 + \int_0^{\tau_0}\hspace{-.2cm} \int_{-M}^M \Gamma(\tau,y)  \frac{f_{2,\ep}(t'_\ep - \ep \tau, x'_\ep - \ep y,u^\ep(t'_\ep - \ep \tau, x'_\ep - \ep y))}{u^\ep(t'_\ep,x'_\ep)}\,dyd\tau \\
    &= -\frac{f_{1, \ep}(t'_\ep,x'_\ep,u^\ep( t'_\ep,x'_\ep))}{u^\ep( t'_\ep,x'_\ep)} - \int_0^{\tau_0}\hspace{-.2cm} \int_{|y|>M} \Gamma(\tau,y)  \frac{f_{2,\ep}(t'_\ep - \ep \tau, x'_\ep - \ep y,u^\ep(t'_\ep - \ep \tau, x'_\ep - \ep y))}{u^\ep(t'_\ep,x'_\ep)}\,dyd\tau\\
    &\leq \eta'- \partial_u f_{1, \ep} (t'_\ep,x'_\ep,0),
\end{align*}
where we used \eqref{eq:2.17} for the first inequality, \eqref{eq:2.18} for the next equality, the fact that $w^\ep - \phi_\ep$ has local max at $(t'_\ep,x'_\ep)$ for the next inequality, \eqref{eq:2.19} for the third inequality, \eqref{LinearModel3} for the next equality, and \eqref{eq:2.20} for the final inequality.  

Since $\partial_u f_{1, \ep} (t'_\ep,x'_\ep,0)=\partial_u f_1\left(\frac{t'_\ep}{\ep},\frac{x'_\ep}{\ep},0\right)$, the above chain of inequalities implies
$$
\left[-\eta' + \underline{R}_2\left(\frac{x_0}{t_0}\right)\right](1-\eta') + \partial_u f_1\left(\frac{t'_\ep}{\ep},\frac{x'_\ep}{\ep},0\right) - \eta' \leq o(1).
$$
Letting $\ep \to 0$, we have
$$
\left[-\eta' + \underline{R}_2\left(\frac{x_0}{t_0}\right)\right](1-\eta')+ \underline{R}_1\left(\frac{x_0}{t_0}\right) - \eta' \leq 0,
$$
which is in contradiction with 
the choice of $\eta'$ in \eqref{eq:2.16}. This completes the proof of the second part of \eqref{eq:1.5}.
\end{proof}

\begin{proof}[Proof of Theorem \ref{thm:1.3b}]
Let $u(t,x)$ be a solution of \eqref{Model} with initial data satisfying (IC$_\infty$). Suppose that (H1)-(H5), and one of (H6), (H6$'$) or (H6$''$) hold. 
Then Proposition \ref{cp:rho} is again applicable. One can then repeat the proof of Theorem \ref{thm:1.3}.
\end{proof}

\section{Applications}\label{sec:app}
Let $u$ be a solution to \eqref{Model} with initial data $\phi$ satisfying (${IC}_\mu$) for some $\mu \in (0,\infty]$. By Theorem \ref{thm:1.3} or Theorem \ref{thm:1.3b}, the (rightward) spreading speed of problem \eqref{Model} is given by the number $\hat{s}^\mu$, which is characterized in \eqref{eq:snlp} as a free-boundary point of certain first order Hamilton-Jacobi equation. In the following, we give the explicit formula for $\hat{s}^\mu$ in two class of environments: the first one being the asymptotically homogeneous environments (see \eqref{eq:R}), the second one being the environments with a single shifting speed (see \eqref{eq:onesup}).

\subsection{Asymptotically homogeneous environments}

We derive the exact spreading speed for asymptotically homogeneous environments, that is, the hypotheses of Theorem \ref{thm:nonkpp1} and, in particular, 
the assumption \eqref{eq:R} are enforced. When $\mu = \infty$ and $f_i$ are independent of $t$ and $x$, the problem was considered in \cite{XuXiao_2016}.

\begin{proof}[Proof of Theorem \ref{thm:nonkpp1}] 
Recall that $\Delta(\lambda,p)$ is given in \eqref{eq:lambdap} and that $\lambda(p)$ is the function that is implicitly defined by $\Delta(\lambda,p) = 0.$ First, 
we observe that $\lambda(p)$ is well-defined, 
since for each fixed $p$,  we have
$\Delta(\lambda, p) \to \mp \infty$ as $\lambda \to \pm \infty$, and that 
$\partial_\lambda \Delta(\lambda,p) \leq -1$ (so that $\lambda\mapsto \Delta(\lambda,p)$ is strictly  decreasing). 

Second, observe that $p \mapsto \lambda(p)$ is even, and strictly convex, i.e.  $\lambda''(p)>0$. Indeed, $\lambda(p)$ is even, since $p \mapsto \Delta(\lambda,p)$ is even.
Furthermore, differentiating the relation $\Delta(\lambda(p),p)=0$ gives
$$
-\partial_\lambda\Delta(\lambda,p) \cdot \lambda'' = \partial_{\lambda\lambda}\Delta(\lambda,p) |\lambda'|^2 + \partial_{pp}\Delta(\lambda,p).
$$
Since $\partial_\lambda\Delta <0$, $\partial_{\lambda\lambda}\Delta \geq 0$ and $\partial_{pp}\Delta>0$, we deduce $\lambda''(p)>0$.

Next, observe that $\lambda(p)/p$ is unbounded as $p \to \infty$. Indeed, using $\Delta(\lambda(p),p)\equiv 0$ and \eqref{eq:lambdap},
$$
\frac{\lambda(p)}{p}=\frac{\lambda(p)+\Delta(\lambda(p),p)}{p} \geq p + \frac{r_1}{p}  \to +\infty \quad \text{ as }p \to \infty.
$$
By evenness, $\lambda(p)/p \to -\infty$ as $p \to -\infty$. 
Recalling that $\lambda'' >0$,  we see that $\lambda': \mathbb{R}\to \mathbb{R}$ is a homeomorphism.  We denote the inverse function of $\lambda'$ to be  $\Psi: \mathbb{R} \to \mathbb{R}$.

Next, observe that there exists a unique positive number $\mu^*$ such that
\begin{equation}\label{qq:3.0}
    p \mapsto \frac{\lambda(p)}{p} \quad \text{ is decreasing in }[0, \mu^*) \text{ and increasing in }(\mu^*,\infty).
\end{equation} 
In particular
\begin{equation}\label{qq:3.1}
    \frac{\lambda(p)}{p} \geq \frac{\lambda(\mu^*)}{\mu^*} \quad \text{ for all }p \geq 0, \quad \text{ and equality holds iff }p= \mu^*.
\end{equation} 
Indeed, let $h(p) = \lambda(p)/p$, then since $\lambda''>0$, it is not difficult to show that
$$
 h'(p_0) = 0  \quad \text{for some $p_0$ implies } \quad h''(p_0) >0.
$$
Observing also that $h(0+)=-\infty$ and $h(+\infty) = \infty$ (since $\lambda'(0) = 0 < \lambda(0)$), we deduce that  $p \mapsto h(p)$ attains its global minimum at a positive number $\mu^*$, and that
\begin{equation}\label{qq:3.2}
 h'(p) <0 \quad \text{ in }  (0,\mu^*),\quad  h'(\mu^*) = 0  \quad \text{ and }\quad h'(p) >0 \quad \text{ in }  (\mu^*,\infty).
\end{equation}
This proves \eqref{qq:3.0}.

Next, we considering the following three cases separately: 
$$
{\rm(i)} \quad \mu \in (0,\mu^*], \quad {\rm(ii)} \quad \mu \in (\mu^*,\infty),\quad \text{ and }{\rm(iii)} \quad \mu = \infty.
$$
Case (i). Let $\mu \in (0,\mu^*]$ be given, and 
define the function
\begin{equation}\label{eq:Hrho2'}
\rho_\mu(s):=\max\{\mu s-\lambda(\mu),0 \} \text{ for }s\in\R_+.
\end{equation}
Then $\rho_\mu$ satisfies \eqref{eq:hjbc} and is a classical solution of  
 \begin{equation}\label{eq:Hrho3}
 \min\left\{\rho,\rho-s\rho'+|\rho'|^2+r_1+r_2\int_{0}^{\tau_0}\int_{\R }\Gamma(\tau, y)e^{(\rho-s\rho')\tau +\rho' y} dy d\tau \right\}=0 
 \end{equation}
in  $[0,\infty) \setminus \{\lambda(\mu)/\mu\}$. Since $\rho_\mu'((\lambda(\mu)/\mu) -) =0 < \rho'_\mu((\lambda(\mu)/\mu)+)=\mu$, we see that $\rho_\mu$ is automatically a viscosity sub-solution of \eqref{eq:Hrho3}
in $(0,\infty)$ (since $\rho_\mu - \phi$ can never attain a strict maximum at $s = \lambda(\mu)/\mu$). To show that it is a viscosity super-solution, suppose $\rho_\mu - \phi$ attains a strict local minimum at $s_0\in (0,\infty)$. If $s_0 \neq \lambda(\mu)/\mu$, then $\rho_\mu$ is differentiable and $\phi'(s_0) = \rho'_\mu(s_0)$ and there is nothing to prove. If $s_0 = \lambda(\mu)/\mu$, then, denoting $\phi'=\phi'(s_0)$, we have $0 \leq \phi' \leq \mu$, and that 
\begin{align*}
   & -(\lambda(\mu)/\mu)\phi'+ |\phi'|^2  - g'(0) + f'(0) \int_0^{\tau_0}\int_\mathbb{R} \Gamma(\tau,y) e^{-(\lambda(\mu)/\mu)\phi'\tau + \phi' y } \,dyd\tau\\
    &= \Delta \left(\frac{\lambda(\mu)}{\mu} \phi', \phi' \right) - \Delta(\lambda(\phi'),\phi')\geq 0,
\end{align*}
where we used the fact that $\lambda \mapsto \Delta(\lambda,p)$ is decreasing in $\lambda$, and that $\frac{\lambda(\mu)}{\mu} \phi' \leq \lambda(\phi')$ (which is due to $0 \leq \phi \leq \mu \leq \mu^*$, and \eqref{qq:3.0}). Hence $\rho_\mu$ is a viscosity solution of \eqref{eq:Hrho3} such that \eqref{eq:hjbc} holds. By the uniqueness result of Proposition \ref{prop:1.2}, we deduce that $\hat\rho^\mu(s) = \rho_\mu(s)$, and hence $\hat{s}^\mu = \lambda(\mu)/\mu$ when $\mu \in (0,\mu^*]$.

Case (ii). Let $\mu \in [\mu^*,\infty)$ be given, define
\begin{equation}\label{eq:Hrho1}
\rho_\mu(s):=\begin{cases}
\mu s-\lambda(\mu) &\text{ for } s\in[\lambda'(\mu),\infty),\\
s\Psi(s)-\lambda(\Psi(s)) & \text{ for } s\in[\lambda'(\mu^*),\lambda'(\mu)),\\
0 & \text{ for } s\in[0,\lambda'(\mu^*)).
\end{cases}
\end{equation}
It is straightforward to check that $\rho_\mu$ is a classical solution of \eqref{eq:Hrho3} in $\mathbb{R} \setminus\{\lambda'(\mu^*)\}.$ 
Using the fact that $\Psi = (\lambda')^{-1}$, it is straightforward to observe that $\rho_\mu$ is continuous at $\lambda'(\mu^*)$ and is differentiable and a classical solution of \eqref{eq:Hrho3} in $[0,\infty) \setminus \{\lambda'(\mu^*)\}$.


Similar as before, $\rho_\mu$ is a viscosity sub-solution of \eqref{eq:Hrho3} in $(0,\infty)$. To show that it is a super-solution as well, it remains to consider the case when $\rho_\mu - \phi$ attains a strict local minimum point at $\lambda'(\mu^*)$. In such an event, $\phi' = \phi'(\lambda'(\mu^*))$ is nonnegative.
Now, at the point $s = \lambda'(\mu^*)$, 
\begin{align*}
   & -\lambda'(\mu^*)\phi'+ |\phi'|^2  - g'(0) + f'(0) \int_0^{\tau_0}\int_\mathbb{R} \Gamma(\tau,y) e^{-\lambda'(\mu^*)\phi'\tau + \phi' y } \,dyd\tau\\
   &\quad = \Delta \left( \lambda'(\mu^*)\phi',\phi'\right) = \Delta \left(\frac{\lambda(\mu^*)}{\mu^*} \phi', \phi' \right) - \Delta(\lambda(\phi'),\phi')\geq 0,
\end{align*}
where we used $\lambda'(\mu^*) = \lambda(\mu^*)/\mu^*$ and $\Delta(\lambda(\phi'),\phi') =0$ in the second equality; \eqref{qq:3.1} and the monotonicity of $\lambda \mapsto \Delta(\lambda,p)$ for the last inequality. By the uniqueness proved in Proposition \ref{prop:1.2}(a), we deduce that the unique viscosity solution of  \eqref{eq:Hrho3} is given by \eqref{eq:Hrho1}. Hence, 
\begin{equation}\label{eq:speedd}
\hat{s}^\mu = \lambda'(\mu^*) = \lambda(\mu^*)/\mu^*= \inf\limits_{p >0} \frac{\lambda(p)}{p}.\end{equation}

Case (iii). Let $\mu \to \infty$ in \eqref{eq:Hrho1}, then the sequence of viscosity solutions $\rho_\mu$ converges, i.e.
$$
\rho_\mu(s) \to \rho_\infty(s):= \max\{ 0, s \Psi(s) - \lambda(\Psi(s)\} \quad \text{ in }C_{loc}([0,\infty)).
$$
By stability property of viscosity solution \cite[Theorem 6.2]{Barles2013}, $\rho_\infty$ is a viscosity solution of \eqref{eq:Hrho3} in $(0,\infty)$. We claim that $\rho_\infty$ satisties \eqref{eq:hjbc} for $\mu=\infty$. Indeed, $\rho_\infty(0) = 0$ and 
$$
\lim\limits_{s \to \infty} \frac{\rho_\infty(s)}{s} \geq \lim\limits_{s \to \infty} \frac{\rho_{\mu}(s)}{s} = \mu \quad \text{ for each }\mu \in [\mu^*,\infty).
$$
Letting $\mu \to \infty$, we verified \eqref{eq:hjbc}. Hence, by the uniqueness result of Proposition \ref{prop:1.2}(b), we conclude that $\rho_\infty$ gives the unique viscosity solution of \eqref{eq:Hrho3}, and thus \eqref{eq:speedd} is valid for $\mu = \infty$ as well. This completes the proof of Theorem \ref{thm:nonkpp1}.
\end{proof}

\subsection{Positive habitat with a single shift}
In this section, we consider environments with a single shifting speed, i.e. 
the hypotheses of Theorem \ref{thm:4} and, in particular, 
the assumption \eqref{eq:onesup} are enforced.


%
For $\mu \in \mathbb{R}$, recall that $\lambda_{-}(\mu)$ and $\lambda_+(\mu)$ are defined by the implicit formula
$$
0 = \Delta_{\pm}(\lambda,\mu) = -\lambda + \mu^2 +R_{1,\pm}+ R_{2,\pm}\int_0^{\tau_0} \int_\mathbb{R} \Gamma(t,y) e^{\mu y - \lambda t}\,dyd\tau. 
$$
Then \eqref{rhoeq1''} can be re-written as 
\begin{equation}\label{rhoeq1'''}
\max\{\rho, \tilde H(s,\rho,\rho') \}=0 \quad \text{ in }(0,\infty)\quad \text{ where }\quad 
\tilde H(s,\rho,\rho') = \begin{cases}
\Delta_-(\rho - s \rho', \rho')  \quad \text{ for }s \leq c_1,\\
\Delta_+(\rho - s \rho', \rho')  \quad \text{ for }s > c_1.\\
\end{cases}
\end{equation}

Since $\lambda_{-}$ (resp. $\lambda_+$) are coercive and strictly convex, $\lambda_-$ (resp. inverse of $\lambda_+$) is a homeomorphism of $\mathbb{R}$ and 
we can similarly define $\Psi_{-}$ (resp. $\Psi_+$) to be the inverse of $\lambda_-$ (resp. inverse of $\lambda_+$).
%
Next, define $(c^*_-,\mu^*_-)$ and $(c^*_+,\mu^*_+)$ by
$$
0<c^*_{\pm} =  \inf_{\mu >0} \frac{\lambda_{\pm}(\mu)}{\mu} = \frac{\lambda_1(\mu^*_\pm)}{\mu^*_\pm}. 
$$
Since $R_{1,-}+R_{2,-}<R_{1,+}+R_{2,+}$,  $\partial_{\lambda}\Delta_{\pm}(\lambda,\mu)<0$ and 
$\Delta_{+}(\lambda_-(\mu),\mu)>\Delta_-(\lambda_-(\mu),\mu)=0$, we see that $\lambda_-(\mu)<\lambda_+(\mu)$ for each $\mu>0$. It then follows that $c^*_-<c^*_+$ and for any $c_1>0$ \begin{equation}\label{eq:rho_c1} 
c_1\Psi_-(c_1)-\lambda_-(\Psi_-(c_1))=\max_{\mu>0}\{c_1\mu-\lambda_-(\mu)\}>\max_{\mu>0}\{c_1\mu-\lambda_+(\mu)\}=c_1\Psi_+(c_1)-\lambda_+(\Psi_+(c_1)).
\end{equation}
Moreover, if $c_1>c^*_+$, then 
\begin{equation}\label{eq:rho_c1'}
   c_1 \Psi_+(c_1)-\lambda_+(\Psi_+(c_1))\ge c_1\mu^*_+-\lambda_+(\mu^*_+)>c^*_+\mu^*_+-\lambda_+(\mu^*_+)=0. 
\end{equation}

In the remainder, we divide the proof of Theorem \ref{thm:4} into the following lemmas.
\begin{lemma}\label{thm:4(i)}
     If $\mu \in (0,\mu^*_+]$, then
 	\begin{equation}
\hat 	s^\mu(c_1)=\begin{cases}
\lambda_+(\mu)/\mu &\text{ if  } c_1\le \lambda_+(\mu)/\mu,\\
	\lambda_-(\underline p(c_1,\mu))/\underline p(c_1,\mu) & \text{ if } c_1>  \lambda_+(\mu)/\mu\text{ and } \underline p(c_1,\mu) <\mu^*_- \\
	c^*_-,& \text{otherwise}
	\end{cases}
	\end{equation}   
where    $\underline{p}(c_1,\mu)$ is the smallest root of
    \begin{equation}\label{eq:up}
    c_1p-\lambda_-(p)=c_1\mu-\lambda_+(\mu).
    \end{equation}
\end{lemma}
We will postpone and sketch the proof once after the more delicate Lemma \ref{lem:muL} is established.
\begin{lemma}\label{lem:muL}
     If $\mu \in (\mu^*_+, \infty)$, then
 	\begin{equation}
\hat 	s^\mu(c_1)=\begin{cases}
c^*_+ &\text{ if  } c_1\le \lambda'_+(\mu^*_+),\\
	\lambda_-(\bar p(c_1))/\bar p(c_1) & \text{ if } \lambda'_+(\mu^*_+)<c_1\le \lambda'_+(\mu)  \text{ and } \bar p(c_1)< \mu^*_-,\\
	\lambda_-(\underline p)/\underline p & \text{ if } c_1>  \lambda'_+(\mu)\text{ and } \underline p(c_1,\mu) <\mu^*_-, \\
	c^*_-,& \text{otherwise}.
	\end{cases}
	\end{equation}  
		where $	c^*_+ = \lambda'_+(\mu^*_+) =\lambda_+(\mu^*_+)/\mu^*_+ $,  $\underline p(c_1,\mu)$ is the smallest root of \eqref{eq:up} and $\bar p(c_1)$ is the smallest root of 
	\begin{equation}\label{eq:op}
	c_1p-\lambda_-(p)=c_1\Psi_+(c_1)-\lambda_+(\Psi_+(c_1))
	\end{equation}
\end{lemma}
Finally, letting $\mu \to \infty$ in the above lemma, we have
\begin{lemma}\label{thm:4(iii)}
     		\begin{equation}\label{eq:onesnlp}
	\hat s^\infty(c_1)=\begin{cases} 
	c^*_+ & \text{ if } c_1\le c^*_+,\\
	\lambda_-(\bar p(c_1))/\bar p(c_1) & \text{ if } c_1\in(c^*_+,\bar c_1),\\
	c^*_-& \text{ if } c_1\ge\bar c_1.
	\end{cases}
	\end{equation}
	where $	c^*_+ = \lambda'_+(\mu^*_+) =\lambda_+(\mu^*_+)/\mu^*_+ $ and 
	$\bar c_1$ is the  unique positive number such that $\bar{p}(\bar c_1)=\mu^*_-$. 
\end{lemma}
\begin{proof}[Proof of Lemma \ref{lem:muL}]
First, we claim that $\underline{p}(c_1,\mu)<\Psi_-(c_1)\wedge\mu$ and $\bar{p}(c_1)<\Psi_+(c_1)\wedge \Psi_-(c_1)$, and that both are increasing in $c_1\in(0,\infty)$.
	
	Indeed, define the auxiliary functions
	$$
	F_1(c_1,\mu,p)=c_1p-\lambda_-(p)-c_1\mu+\lambda_+(\mu)\quad \text{ and }\quad F_2( c_1,p)=c_1p-\lambda_-(p)-c_1\Psi_+(c_1)+\lambda_+(\Psi_+(c_1)). $$
	Then $F_i$ is increasing in $p\in[0,\Psi_-(c_1)]$ and decreasing in $p\in[\Psi_-(c_1),\infty)$, and
	\begin{align*}
	& F_1(c_1,\mu,\mu)=-\lambda_-(\mu)+\lambda_+(\mu)>0,\quad F_1(c_1,\mu,0)<0,\\
	& F_2(c_1,\Psi_+(c_1))=-\lambda_-(\Psi_+(c_1))+\lambda_+(\Psi_+(c_1))>0,\quad  F_2(c_1,0)<0.
	\end{align*}
	It then follows that the smallest roots $\underline{p}(c_1,\mu)\in(0,\mu)$ and $\bar{p}(c_1)\in(0,\Psi_+(c_1))$. Moreover,
	\begin{align}\label{eq:DF}
&\partial_p F_1(c_1,\mu,\underline p)>0, \quad \partial_{\mu} F_1(c_1,\mu,\underline p)=-c_1+\lambda'_+(\mu),\quad \partial_c F_1(\mu, c_1,\underline{p})=\underline{p}-\mu<0.\\
&\partial_p F_2(c_1,\bar p)>0, \quad \partial_c F_2( c_1,\bar{p})=\bar{p}-\Psi_+(c_1)<0\nonumber
	\end{align}   Therefore, $\underline{p}(c_1,\mu)$ is increasing in $c_1$, and  increasing in $\mu$ provided $c_1>\lambda'_+(\mu)$.  
	$\bar p(c_1)$ is increasing in $c_1$. Since $F_2(c^*_+,\mu^*_-)=\mu^*_-(c^*_+-c^*_-)>0$, we obtain $\bar p(c^*_+)<\mu^*_-$ and $\bar p(+\infty)=+\infty$. There exists a unique number $\bar c_1$ such that $\bar p(\bar c_1)=\mu^*_-.$
	
	Second, if $c_1\ge\lambda'_+(\mu)$, we claim that $\underline{p}(c_1,\mu)=\mu^*_-$ defines a decreasing function $c_1=g(\mu)$ for $\mu^*_-< \mu\le\Psi_+(\bar c_1)$ and $\bar{c}_1=g(\Psi_+(\bar{c}_1))$.
	
	In fact, $c_1=g(\mu)$ solves implicitly from $F_1(c_1,\mu,\mu^*_-)=0$.
	It is decreasing for $\mu\in (\mu^*_-,\Psi_+(\bar c_1]$ due to \eqref{eq:DF}.
	A direct computation gives $F_1(\bar c_1, \Psi_+(\bar{c}_1),\mu^*_-)=F_2(\bar c_1, \mu^*_-)=0$.
		
Next,  we divide the proof into the following five mutually exclusive cases:
(i) $c_1\le \lambda'_+(\mu^*_+)$ ; (ii) $\lambda'_+(\mu^*_+)<c_1\leq \lambda'_+(\mu)$ and $\bar{p}(c_1)<\mu^*_-$; (iii) $c_1>\lambda'_+(\mu)$ and $\underline{p}(c_1,\mu)<\mu^*_-$; (iv) $\lambda'_+(\mu^*_+)<c_1\leq \lambda'_+(\mu)$ and $\bar{p}(c_1)\geq\mu^*_-$; (v) $c_1>\lambda'_+(\mu)$ and $\underline{p}(c_1,\mu)\geq \mu^*_-$.

Case (i). Since $c_1\le c^*_+$, we can directly verify that the formula \eqref{eq:Hrho1} (with $\lambda,\lambda',\Psi$ replaced by $\lambda_+,\lambda'_+,\Psi_+$) defines a viscosity solution of \eqref{rhoeq1''} in $[0,\infty)$ with $R_i$ given by \eqref{eq:onesup}, which satisfies the boundary conditions \eqref{eq:hjbc}. Hence, the spreading speed $\hat{s}^\mu$ coincides with the homogeneous spreading speed $c^*_+$ in the "$+$" environment.

Case (ii). Note that $\lambda'_+(\mu^*_+)<c_1\leq \lambda'_+(\mu)$. Define the function
\begin{equation}\label{eq:caseiii}
\rho_\mu(s):=\begin{cases}
\mu s-\lambda_+(s)&\text{ for } s\in[\lambda'_+(\mu),+\infty),\\
s\Psi_+(s)-\lambda_+(\Psi_+(s)) &\text{ for } s\in[c_1,\lambda'_+(\mu)),\\
\max\{\bar{p}s-\lambda_-(\bar{p}),0\} &\text{ for } s\in[0, c_1).\\
\end{cases}
\end{equation}
Then $\rho_\mu$ satisfies \eqref{eq:hjbc} and is a classical solution of  \eqref{rhoeq1''} with $R_i(s)$ given in \eqref{eq:onesup} in $[0,\infty)$ except $\{c_1,\lambda_-(\bar{p})/\bar{p}\}$. Since $\rho'_{\mu}(c_1-)=\bar{p}<\Psi_+(c_1)=\rho'_{\mu}(c_1+)$ and $\rho'_{\mu}(\lambda_-(\bar p)/\bar p-)=0<\rho'_{\mu}(\lambda_-(\bar p)/\bar p+)=\bar{p}$, we conclude that $\rho_\mu$ is automatically a viscosity sub-solution in $(0,\infty)$.
It suffices to show it is also a viscosity super-solution, it suffices to consider the two points $s=c_1$ and $s=\lambda_-(\bar{p})/\bar{p}$ where the $\rho_\mu$ may not be differentiable. 
Since $\overline p < \mu^*_-$, the latter point can be treated as in case (i) of proof of Theorem  \ref{thm:nonkpp1}. Thus, it suffices to consider the case when $\rho_\mu - \phi$, for some test function $\phi$, attains a strict local minimum at the point $s=c_1$.
Now, at the point $s=c_1$,
\begin{align*}
& \rho_\mu(c_1)-c_1\phi'+|\phi'|^2+R_{1,+}+R_{2,+}\int_{0}^{\tau_0}\int_{\R }\Gamma(\tau, y)e^{( \rho_\mu(c_1)-c_1\phi')\tau+\phi' y} dy d\tau\\
&=\Delta_+(c_1\phi'-c_1\Psi_+(c_1)+\lambda_+(\Psi_+(c_1)), \phi')-\Delta_+(\lambda_+(\phi'),\phi')\ge0,
\end{align*}
holds for $\phi'\in[\bar p,\Psi_+(c_1)]$,
where we used $\Delta_+(\lambda_+(p),p) = 0$ for all $p$ for the first equality, and 
the fact that $\Delta_+(\lambda,p)$ is decreasing in $\lambda$ and that $c_1\phi'-c_1\Psi_+(c_1)+\lambda_+(\Psi_+(c_1)) \leq \lambda_+(\phi')$ (see the last equality of \eqref{eq:rho_c1}) for the last inequality. Hence, $\hat s^\mu=\lambda_-(\bar p)/\bar p$.

Case (iii). Define the function
\begin{equation}
\rho_\mu(s):=\begin{cases}
\mu s-\lambda_+(s) &\text{ for } s\in[c_1,+\infty),\\
\max\{\underline{p}s-\lambda_-(\underline{p}),0\} &\text{ for } s\in[0,c_1).
\end{cases}
\end{equation}
A direction computation gives $\rho_\mu$ is a classical solution of \eqref{rhoeq1''}  except $\{c_1,\lambda_-(\underline{p})/\underline{p}\}$. Similar as above, it is a viscosity sub-solution of \eqref{rhoeq1''} in $(0,\infty)$. To show it is also a super-solution, it suffices to consider the two points $s=c_1$ and $s=\lambda_-(\underline p)/\underline p$ where $\rho_\mu$ is not differentiable. Since $\underline p < \mu^*_-$, the latter point can be treated as in case (i) of proof of Theorem \ref{thm:kpp1}. Thus, it suffices to consider the case when $\rho_\mu - \phi$, for some test function $\phi$, attains a strict local minimum at the point $s=c_1$. At that point, we have 
\begin{align*}
&\rho_\mu(c_1)-c_1\phi'+|\phi'|^2+R_{1,+}+R_{2,+}\int_{0}^{\tau_0}\int_{\R }\Gamma(\tau, y)e^{(\rho_\mu(c_1)-c_1\phi')\tau+\phi' y} dy d\tau\\
&=\Delta_+(c_1\phi'-c_1\mu+\lambda_+(\mu), \phi')-\Delta_+(\lambda_+(\phi'),\phi')\ge0,
\end{align*}
holds for $\phi'(c_1):=\phi'\in[\underline{p},\mu]$, 
where we used the fact that $\Delta_+(\lambda,p)$ is decreasing in $\lambda$, and $c_1 s-\lambda_+(s)$ is increasing on $[0,\mu]$. Hence, $\hat s^\mu=\lambda_-(\underline{p})/\underline{p}$.

Case (iv). 
Denote the function 
\begin{equation}\label{eq:caseiv}
\overline\rho_\mu(s):=\begin{cases}
\mu s-\lambda_+(s)&\text{ for } s\in[\lambda'_+(\mu),+\infty),\\
s\Psi_+(s)-\lambda_+(\Psi_+(s)) &\text{ for } s\in[c_1,\lambda'_+(\mu)),\\
\bar{p}s-\lambda_-(\bar{p}) &\text{ for } s\in[\lambda'_-(\bar p), c_1],\\
s\Psi_-(s)-\lambda_-(\Psi_-(s)) &\text{ for } s\in[c^*_- \lambda'_-(\bar p)],\\
0 &\text{ for } s\in[0, c^*_-].
\end{cases}
\end{equation}
 Similarly as before, we see that \eqref{eq:caseiv} is a viscosity sub-solution. Furthermore, one might directly adapt the argument in Case (i) of Theorem  \ref{thm:nonkpp1} to obtain it is also a viscosity super-solution. Therefore, we concludes $\hat s^\mu=c^*_-$.

Case (v). Now it is standard to check that 
$\rho_\mu\in C([0,\infty)]$ defined by
\begin{equation}\label{eq:caseii}
\rho_\mu(s):=\begin{cases}
\mu s-\lambda_+(s) &\text{ for } s\in[c_1,+\infty),\\
\underline{p}s-\lambda_-(\underline{p}), &\text{ for } s\in[\lambda'_-(\underline{p}),c_1],\\
s\Psi_-(s)-\lambda_-(\Psi_-(s))  &\text{ for } s\in[c^*_-,\lambda'_-(\underline{p})],\\
0&\text{ for } s\in[0,c^*_-].
\end{cases}
\end{equation}
is a unique viscosity solution of \eqref{rhoeq1''}-\eqref{eq:hjbc}. Hence, $\hat s^\mu=c^*_-$.
\end{proof}
\begin{proof}[Sketch proof of Lemma 3.5]
First, we emphasize that if $\mu^*_+\le \mu^*_-$, then $\bar p(c_1,\mu)<\mu^*_-$ is always valid due to $\bar{p}<\mu$, so there will be only two mutually exclusive cases for $\mu\in(0,\mu^*_+)$, namely, 

(i) $c_1\le \lambda_+(\mu)/\mu$;  
(ii) $c_1>  \lambda_+(\mu)/\mu\text{ and } \underline p(c_1,\mu) <\mu^*_-$.

Case (i). It follows from an easy modification of Case (i) in Theorem  \ref{thm:nonkpp1}.

Case (ii). This is already done by Case (iii) in the  proof of Lemma \ref{lem:muL}.

In addition, when $\mu^*_-<\mu^*_+$. 
Case (iii)  $c_1>  \lambda_+(\mu)/\mu\text{ and } \underline p(c_1,\mu) \ge\mu^*_-$ does exist. But then, one could check this by the similar strategy in case (v) in proof of Lemma \ref{lem:muL}.
\end{proof}
\begin{remark}
	If $c_1\le 0$, then the result is the same as that of homogeneous environment.
	When $\mu^*_-\ge\mu^*_+>\mu$, $\frac{\lambda_-(\underline{p}(c_1,\mu))}{\underline{p}(c_1,\mu)}\to \frac{\lambda_-(\mu)}{\mu}$ as $c_1\to \infty$.
\end{remark}

\subsection{Proof of Theorems \ref{thm:kpp1} and \ref{thm:kppp}}

\begin{proof}[Proof of Theorem \ref{thm:kpp1}]
Let $r_0 >0$ be a given positive constant. 
We apply Theorem \ref{thm:nonkpp1} with $f_1(t,x,u)=u(r(t,x)-u)$ and $f_2\equiv 0$. Then (H1)-(H6) are satisfied with
$$\underline{R}_1(s) = r_0 \quad \text{ a.e. in }(0,\infty), \quad 
R_1 \equiv r_0\quad \text{ and } \quad 
R_2 \equiv \underline{R}_2 \equiv 0.  
$$
And \eqref{eq:lambdap} takes the form $\Delta(\lambda,p)=-\lambda + p^2 + r_0 =0$. Hence $\lambda(p) = p^2 + r_0$, and the minimum point $\mu^*$ of $\lambda(p)/p$ is $\sqrt{r_0}$. Hence, the formula \eqref{eq:1.24b} follows from \eqref{eq:R22}. 
\end{proof}

\begin{proof}[Proof of Theorem \ref{thm:kppp}]
We take $f_2\equiv0$ and $R_{1,+}=r_2>R_{1,-}=r_1>0$ in Theorem \ref{thm:4}. Now we see  
\begin{equation}
\lambda_-(p)=p^2+r_1, \lambda_+(p)=p^2+r_2, r_1<r_2
\end{equation}
Moreover,
$c^*_-=2\sqrt{r_1}<c^*_+=2\sqrt{r_2}$ and $\mu^*_-=\sqrt{r_1}<\mu^*_+=\sqrt{r_2}$. First, we derive \eqref{123} and \eqref{1234} from \eqref{eq:1.18}. Note that  $\underline{p}$ is defined by the smallest root of \eqref{eq:up'}, i.e.
$c_1p-p^2-r_1=c_1\mu-\mu^2-r_2,$ then we get
$$\underline{p}=\frac{c_1-\sqrt{(c_1-2\mu)^2+4(r_2-r_1)}}{2}.
$$
When $\mu \leq \sqrt{r_1}$, then $\underline{p} \leq \mu^*_-$ and \eqref{123} follows from the first two alternatives in \eqref{eq:1.18}. 
Note that the third alternative in \eqref{eq:1.18} holds only when $\mu \in (\sqrt{r_1},\sqrt{r_2}]$, when 
$$
c_1 \geq \frac{\mu^2+r_2-2r_1}{\mu-\sqrt{r_1}}=\mu+\sqrt{r_1}+\frac{r_2-r_1}{\mu-\sqrt{r_1}}.
$$
Hence \eqref{123} and \eqref{1234} follows from \eqref{eq:1.18}, by noting that $c^*_- = 2\sqrt{r_1},$ 
\begin{equation}\label{eq:33}
\frac{\lambda_+(\mu)}{\mu} = \mu + \frac{r_2}{\mu}\quad \text{ and }\quad \frac{\lambda_-(\underline p)}{\underline p} = \frac{c_1-\sqrt{(c_1-2\mu)^2+4(r_2-r_1)}}{2} + \frac{2r_1}{c_1-\sqrt{(c_1-2\mu)^2+4(r_2-r_1)}}.
\end{equation}

Next, let $\mu \geq \sqrt{r_2}$. We derive \eqref{12345} and \eqref{12} from \eqref{eq:1.20}. Here we note that $\overline p$ is the smallest root of \eqref{eq:op'}, i.e.  $c_1p-p^2-r_1=\frac{c^2_1}{2}-\frac{c^2_1}{4}-r_2.$ Hence, 
$$\bar p=\frac{c_1}{2}-\sqrt{r_2-r_1}.
$$
And we observe that the second alternative in \eqref{eq:1.20} happens precisely when $\mu \in [\sqrt{r_2}, \sqrt{r_1} + \sqrt{r_2-r_1}]$ and $c_1 \in (2\sqrt{r_2}, 2\sqrt\mu]$. This divides \eqref{eq:1.20} into the two cases \eqref{12345} and \eqref{12}. Note that $c^*_- = 2\sqrt{r_1},$ $c^*_+ = 2\sqrt{r_2}$, $\lambda$, ${\lambda_-(\underline p)}/{\underline p}$ is given in \eqref{eq:33}, and 
$$
\frac{\lambda_-(\bar p)}{\bar p} =\frac{c_1}{2}-\sqrt{r_2-r_1} + \frac{r_1}{\frac{c_1}{2}-\sqrt{r_2-r_1}}.
$$
We omit the details.
\end{proof}

\appendix 
\section{Proof of Proposition \ref{cp:rho1}}\label{sec:A}




In this section, we proof Proposition \ref{cp:rho1} by applying the comparison result in \cite[Theorem A.1]{LLL_2020b}, which was inspired by the arguments developed by Ishii \cite{Ishii1997} and Tourin \cite{Tourin1992}. Consider the following Hamilton-Jacobi equation:
\begin{equation}\label{eq:hj1}
\min\{w_t+\tilde{H}(t,x,D_x w), w-Lt\}=0 \quad \text{in} \,\,\Omega \quad \text{ for some }L \in \mathbb{R}.
\end{equation}


\begin{definition}\label{def:AA}
We say that a  lower semicontinuous (lsc) function $w$ is  a  viscosity super-solution of \eqref{eq:hj1} if $w-Lt\geq 0$ in $ \Omega$, and for all test functions $\varphi\in C^\infty(\Omega)$, if $(t_0,x_0)\in \Omega$ is a strict local minimum point of $w-\varphi$, then
$$\partial_t\varphi(t_0,x_0)+\tilde{H}^*(t_0,x_0,D_x\varphi(t_0,x_0))\geq 0$$
holds; A upper semicontinuous (usc) function $w$ is a  viscosity sub-solution of \eqref{eq:hj1} if for all test functions $\varphi\in C^\infty(\Omega)$,  if $(t_0,x_0)\in \Omega$  is a strict local maximum point of $w-\varphi$ such that $w(t_0,x_0) - Lt_0>0$, then
$$\partial_t\varphi(t_0,x_0)+\tilde{H}_*(t_0,x_0,D_x\varphi(t_0,x_0))\leq 0$$
holds. Finally, $w$ is a  viscosity solution  of \eqref{eq:hj1} if and only if $w$ is simultaneously a viscosity super-solution and a viscosity sub-solution of \eqref{eq:hj1}.
\end{definition}


Let $\Omega$ be a domain in  $[0,T)\times \mathbb{R}^N$ with some given $T>0$. 
We impose additional assumptions the Hamiltonian $\tilde{H}:\Omega\times\mathbb{R}^N\rightarrow \mathbb{R}$. Namely,
 for each $R>0$ there exists a continuous function $\omega_R: [0,\infty) \to [0,\infty)$ such that $\omega_R(0) =0$ and $\omega_R(r)>0$ for $r>0$, such that the following holds:

\begin{enumerate}
	\item [(A1)] For each $(t_0,x_0) \in \Omega$, $p\mapsto \tilde{H}(t_0,x_0,p)$ is a continuous function from $\mathbb{R}^N$ to $\mathbb{R}$;


	\item[(A2)] For each $R>0$ and  $(t_0,x_0)\in \Omega \cap[ (0,T)\times B_R(0)]$, there exist a constant $\delta_0>0$ and a unit vector $(h_0,k_0)\in \mathbb{R}\times \mathbb{R}^N$ such that
	$$\tilde{H}(s,y,p)-\tilde{H}(t,x,p)\le \omega_R((|x-y|+|t-s|)(1+|p|))$$
	for $t,x,s,y,p$ such that  $p\in \mathbb{R}^{N}$, and 
	\begin{equation}\label{eq:A2}
	0<\|(t,x)-(t_0,x_0)\|+\|(s,y)-(t_0,x_0)\|<\delta_0,\quad \left\|\frac{(t-s,x-y)}{\|(t-s,x-y)\|}-(h_0,k_0) \right\|<\delta_0.
	\end{equation}
	\item [(A3)] There is a constant $M\ge0$ such that for each $\lambda\in[0,1)$ and $x_0\in \mathbb{R}^N$, there exist  constants  $\bar\epsilon(\lambda,x_0)>0$ and $\bar C(\lambda,x_0)>0$ such that 
	%
	$$\tilde{H}\left(t,x,\lambda p-\frac{\ep(x-x_0)}{|x-x_0|^2+1}\right)-M\le \lambda( \tilde{H}(t,x,p)- M)+\ep \bar C(\lambda,x_0)$$
	for all $(t,x,p) \in \Omega \times \mathbb{R}^N$ and $\epsilon \in [0,\bar\epsilon(\lambda,x_0)]$.

%

\end{enumerate}
\begin{thm}\label{CP}
	Suppose that $\tilde{H}$ satisfies (A1)--(A3). Let $\bar w$ and $\underline{w}$ be a pair of super- and sub-solutions of \eqref{eq:hj1} such that $\bar{w}\ge\underline{w}$ on  $\partial_p\Omega$, then $\bar{w}\ge\underline{w}$ in $\Omega.$
\end{thm}
\begin{proof}
This is \cite[Theorem A.1]{LLL_2020b}, by taking the set $\Gamma$ to be the entire $\Omega$, the hypotheses (A1)--(A4) therein become (A1)--(A3) here.
\end{proof}

For our purpose, let $$H(s,q,p) = q + |p|^2 + R_1(s) + R_2(s)\int_0^{\tau_0}\int_\mathbb{R} \Gamma(\tau,y) e^{\tau q + y p}\,dyd\tau.$$ 
Since $H$ is strictly increasing in $q$, we define $\tilde{H}(s,p)$ by 
\begin{equation}\label{eq:AA1}
H(s,q,p) = 0 \quad \text{ if and only if }\quad -q = \tilde{H}(s,p).
\end{equation}
Define $H_*$ and $\tilde{H}_*$ to be the lower envelopes of $H$ and $\tilde{H}$ respectively:
$$
H_{*}(s,q,p) = \liminf_{(s',q',p') \to (s,q,p)} H(s',q',p'),\quad \text{ and } \quad \tilde{H}_{*}(s,p) = \liminf_{(s',p') \to (s,p)} \tilde{H}(s',p') .
$$
Similarly, define the upper envelopes $H^*$ of $H$ and $\tilde{H}^*$ of $\tilde{H}$ by replacing $\liminf$ by $\limsup$ in the above.
\begin{lemma}\label{lem:AA1} We show that \eqref{eq:AA1} holds for the lower and upper envelopes as well, i.e. 
\begin{equation}\label{eq:AA2}
\begin{cases}
H_*(s_0,q_0,p_0) \leq  0   \quad \text{ if and only if }\quad q_0 + \tilde{H}_*(s_0,p_0) \leq 0,\\
H^*(s_0,q_0,p_0) \geq 0  \quad \text{ if and only if }\quad q_0 + \tilde{H}^*(s_0,p_0) \geq 0\\
\end{cases}
\end{equation}
\end{lemma}
\begin{proof}
We only show the first part of \eqref{eq:AA2}, since the latter part is analogous. Let $q_0 = -\tilde{H}_*(s_0,p_0)$. By monotonicity of $H$ in $q$, it remains to show that $H_*(s_0,q_0,p_0) = 0$.

First, choose $(s_n,p_n) \to (s_0,p_0)$ such that $q_n:= -\tilde{H}(s_n,p_n) \to q_0$. By definition of $\tilde{H}$, we have $H(s_n,q_n,p_n) = 0$. Taking $n \to \infty$, we have
\begin{equation}\label{eq:AA3}
0 \geq H_*(s_0,q_0,p_0).
\end{equation}

Next, choose another sequence $(s'_n,q'_n,p'_n) \to (s_0,q_0,p_0)$ such that 
$$
H(s'_n,q'_n,p'_n) \to H_*(s_0,q_0,p_0).
$$
Then by the monotonicity of $H$ in $q$, we have
$$
0 = H(s'_n, - \tilde{H}(s'_n,p'_n), p'_n) \leq  H(s'_n, - \tilde{H}_*(s'_n,p'_n), p'_n).
$$
and hence
\begin{equation}\label{eq:AA3b}
    0 \leq H(s'_n, q'_n, p'_n) + \left[ H(s'_n, - \tilde{H}_*(s'_n,p'_n), p'_n) - H(s'_n, q'_n, p'_n)\right].
\end{equation}
We claim that
\begin{equation}\label{eq:AA4}
\limsup_{n \to \infty} \left[ H(s'_n,- \tilde{H}_*(s'_n,p'_n),p'_n) - H(s'_n,q'_n,p'_n)\right] \leq 0.
\end{equation}
Indeed, this is due to $\limsup_{n \to \infty} [- \tilde{H}_*(s'_n,p'_n)]\leq q_0$ and that $(s'_n,q'_n,p'_n) \to (s_0,q_0,p_0)$, and that $H(s,q,p)$ is continuous and monotone in $q$. Having proved \eqref{eq:AA4}, we can  take $n \to \infty$ in \eqref{eq:AA3b} to get
$0 \leq H_*(s_0,q_0,p_0).$ Combining with \eqref{eq:AA3}, the first part of \eqref{eq:AA2} is proved.
\end{proof}
\begin{lemma}\label{lem:AA3}
Let $\tilde{H}(s,p)$ be given in \eqref{eq:AA1}. Then $\tilde{H}$ is convex in $p$, and hypothesis {\rm(A3)} holds.
\end{lemma}
\begin{proof}We first prove the convexity of $\tilde{H}$. It suffices to show that 
\begin{equation}\label{eq:AA5}
\frac{1}{2}\tilde{H}(s,p_1) + \frac{1}{2} \tilde{H}(s,p_2) \geq \tilde{H}\left(s, \frac{p_1+p_2}{2}\right) \quad \text{ for any }s, p_1, p_2.
\end{equation}
For $i=1,2$, denote $q_i = \tilde{H}(s,p_i)$, then by the convexity of $H(s,q,p)$ in $(q,p)$, then \begin{align*}
0=\frac{1}{2} H(s,-q_1,p_1) + \frac{1}{2} H(s,-q_2,p_2) \geq H\left(s, -\frac{q_1+q_2}{2},\frac{p_1+p_2}{2}\right).
\end{align*}
By the monotonicity of $H$ in $q$, we may compare the above with $H\left(s, - \tilde{H}\left(s, \frac{p_1+p_2}{2}\right),\frac{p_1+p_2}{2}\right)=0$ and deduce $-\frac{q_1+q_2}{2} \leq - \tilde{H}\left(s, \frac{p_1+p_2}{2}\right).$ This proves \eqref{eq:AA5}.

Next, the hypothesis (A3) follows as a consequence of the convexity. Indeed,
\begin{align*}
    \tilde{H}(x/t, \lambda p - \ep \psi'(x-x_0)) &\leq\lambda \tilde{H}(x/t,p) + \epsilon \tilde{H}(x/t,\psi'(x-x_0)) + (1-\lambda - \epsilon) \tilde{H}\left(x/t, 0\right)   \\
    &\leq  \lambda \tilde{H}(x/t,p) + (1-\lambda) M + \ep \overline{C}(\lambda,x_0),
\end{align*}
where $\psi'(x-x_0) = \frac{x-x_0}{|x-x_0|^2 + 1}$, $M = \sup \tilde{H}(x/t,0)$ and $\overline{C}(\lambda,x_0)= \sup \tilde{H}\left(x/t, \psi'(x-x_0)\right).$
\end{proof}

\begin{proof}[Proof of Proposition \ref{cp:rho1}]
Let $\underline\rho$ and $\overline\rho$ be respectively sub- and super-solutions of 
$$
\min\{H(s,\rho - s \rho', \rho'),\rho\} = 0 \quad \text{ for }s \in (0,\infty).
$$
It follows from Lemma \ref{lem:AA1} that they are respectively sub- and super-solutions of
$$
\min\{\rho- s\rho' + \tilde{H}(s,\rho'),\rho\}=0\quad \text{ for }s \in (0,\infty).
$$
Define 
$$
\underline{w}(t,x) = t \underline\rho(x/t) \quad \text{ and }\quad \overline{w}(t,x) = t \overline\rho(x/t),
$$
one can argue similarly as in the proof of Lemma \ref{lem:samehj}, to show that $\underline{w}$ and $\overline{w}$ are respectively sub- and super-solutions of 
\begin{equation}\label{eq:AA00}
\min\{w_t + \tilde{H}(x/t,w_x), w \} = 0\quad \text{ for }(t,x) \in (0,\infty)\times(0,\infty).
\end{equation}

To apply Theorem \ref{CP}, we need to verify the boundary conditions. Now, 
$$\underline{w}(t,0)=t\underline{\rho}(0)\le t\overline{\rho}(0)=\overline{w}(t,0) \quad \text{ for each }t>0,$$
and for each $x>0$, 
$$\underline{w}(0,x)\le\limsup_{t\to0^+}\left[t\underline{\rho}\left(\frac{x}{t}\right)\right]=x\lim_{s\to\infty}\frac{\underline{\rho}(s)}{s}\le x \liminf_{s\to\infty}\frac{\overline{\rho}(s)}{s}=\liminf_{t\to0^+}\left[t\overline{\rho}\left(\frac{x}{t}\right)\right]\le\overline{w}(0,x). $$
Moreover, for $t>0$ and $x>0$, 
$\underline{w}(t,x)\le x \sup\limits_{s>0}\frac{\underline{\rho}(s)}{s}$. Since the supremum is finite (see Lemma \ref{lem:2.3}(iv)), this means that
 $\underline{w}(t,x)$ is continuous at $(0,0)$ and $\underline{w}(0,0)=0\le \overline{w}(0,0)$. 

It remains to verify the hypotheses (A1)--(A3). Now, (A1) is obvious, and (A3) is verified in Lemma \ref{lem:AA3}. We claim that (H4) implies (A2). We divide into the following cases:
\begin{itemize}
    \item[{\rm(i)}] $R_1$ and $R_2$ are both non-increasing, or both non-decreasing.
    \item[{\rm(ii)}] $R_1$ is continuous, and $R_2$ is monotone.
        \item[{\rm(iii)}]  $R_2$ is piecewise constant, and the functions $R_1$ and $R_1 + R_2$ are locally monotone.
\end{itemize}
First, we consider the case (ii), and assume for the moment that $R_2$ is non-decreasing. Fix $(t_0,x_0)$ and let $(h_0,k_0) =\frac{(-x_0,t_0)}{\|(-x_0,t_0)\|}$. Then for $(t,x,s,y)$ satisfying \eqref{eq:A2}, we have $\frac{x}{t} > \frac{y}{s}$, so that
\begin{equation}\label{eq:AA7a}
0 = H \left(\frac{y}{s},- \tilde{H}\left(\frac{y}{s},p\right),p\right) 
\leq R_1\left(\frac{y}{s}\right)-R_1\left(\frac{x}{t}\right) + H\left(\frac{x}{t}, -\tilde{H}\left(\frac{y}{s},p\right),p\right).
\end{equation}
Hence,
\begin{align}
    \tilde{H}\left(\frac{y}{s},p\right) - \tilde{H}\left(\frac{x}{t},p\right)&\leq  H\left(\frac{x}{t},-\tilde{H}\left(\frac{x}{t},p\right),p\right) - H\left(\frac{x}{t},-\tilde{H}\left(\frac{y}{s},p\right),p\right) \notag\\
    &\leq R_1\left( \frac{y}{s}\right) - R_1\left(\frac{x}{t} \right)  \leq \omega_R(|x-y|+|t-s|), \label{eq:AA7}
\end{align}
where we used the fact that $\partial_q H \geq 1$ for the first inequality, and $H\left(\frac{x}{t},-\tilde{H}\left(\frac{x}{t},p\right),p\right)=0$ for the second inequality, and the fact that $R_1$ is continuous (and that $t,s$ are bounded away from zero) for the last inequality. 

In case $R_2$ is non-decreasing, the proof for case (i) is the same as case (ii), where the right hand side of \eqref{eq:AA7} is replaced by $0$, since $R_1$ has the same monotonicity of $R_2$. The proof for cases (i) and (ii) when $R_2$ is non-increasing is similar, and we omit the details.

It remains to verify (A2) for the case (iii). Since $R_2$ is piecewise constant, there exists a countable set $\{s_k\}$ such that $R_2$ is constant in each open interval in $\mathbb{R} \setminus \{s_k\}$. To verify (A3), suppose first $(t_0,x_0)$ is given so that $s_0:=x_0/t_0 \not\in \{s_k\}$. Then by the local monotonicity of $R_1$, there exists a unit vector $(h_0,k_0) = (-x_0,t_0)$ or $(x_0,-t_0)$ such that for $(t,x,s,y)$ satisfying \eqref{eq:A2}, we have
\begin{equation}
R_2\left( \frac{x}{t}\right) = R_2\left( \frac{y}{s}\right)\quad \text{ and }\quad   R_1\left( \frac{y}{s}\right)-   R_1\left( \frac{x}{t}\right)   \leq \omega_R(|x-y|+|t-s|).
\end{equation}
Then again \eqref{eq:AA7a} and \eqref{eq:AA7} holds. 

It remains to consider the case when $s_0 = s_k$ for some $k$, i.e. $R_2$ has a jump discontinuity. 
Assume, for definiteness, that $R_2(s_0+) > R_2(s_0-)$.
First, we claim that there is $\delta_1>0$ such that 
\begin{equation}\label{eq:AA8}
R_2(s') = R_2(s_0+) \quad \text{ for }s' \in [s_0,s_0+\delta_1),\quad R_2(s')= R_2(s_0-) \text{ for }s' \in (s_0-\delta_1,s_0).
\end{equation}
Indeed, $R_2$ is piecewise constant, so the above holds for $s'$ close to but not equal to $s_0$. Next, the fact that $R_1 + R_2$ is locally monotone implies that
\begin{equation}\label{eq:AA8b}
R_1(s_0) + R_2(s_0) \leq \limsup_{s' \to s_0} (R_1 + R_2)(s') \leq \limsup_{s' \to s_0} R_1(s') + \limsup_{s' \to s_0} R_2(s').
\end{equation}
Since $R_i$ are u.s.c., we have $R_i(s_0) \geq \limsup\limits_{s' \to s_0} R_i(s)$ for $i=1,2$. Substituting these into \eqref{eq:AA8b}, we have 
\begin{equation}\label{eq:A.16a}
    R_i(s_0)=\limsup\limits_{s' \to s_0} R_i(s')\quad \text{ for }\quad i=1,2
\end{equation} In view of the assumption $R_2(s_0+) > R_2(s_0-)$, we must have $R_2(s_0) = R_2(s_0+)$. We have proved \eqref{eq:AA8}.

Next, we claim that 
\begin{equation}\label{eq:AA9}
\lim_{\delta \to 0} \mathop{\sup\limits_{|s_i-s_0|<\delta}}_{s_1<s_2} (R_1(s_1)-R_1(s_2)) \leq 0.
\end{equation}
Otherwise, by the fact that $R_1$ is locally monotone (see Definition \ref{def:lc}), there exist $\delta_0>0$ and sequences  $s_{1,j} \to s_0$ and $s_{2,j} \to s_0$ such that $s_{1,j}<s_{2,j}$ and
\begin{equation}\label{eq:AA10}
\lim_{\delta \to 0} \mathop{\inf\limits_{|s_i-s_0|<\delta}}_{s_1<s_2} (R_1(s_1)-R_1(s_2)) \geq 0,\quad\text{ and }\quad R_1(s_{1,j}) > R_1(s_{2,j}) + \delta_0.
\end{equation}
By the first part of \eqref{eq:AA10}, we deduce that $\liminf\limits_{s' \to s_0 -} R_1(s') \geq R_1(s_0)$. Since $R_1$ is u.s.c., we deduce that  $R_1$ is left continuous at $s_0$, i.e. $\lim\limits_{s' \to s_0 -} R_1(s') = R_1(s_0)$. In view of the second part of \eqref{eq:AA10}, it is impossible for both $s_{1,j}$ and $s_{2,j}$ to be less than equal to
$s_0$ for $j \gg 1$. i.e. we have $s_{2,j}> s_0$ for $j \gg 1$. Using also \eqref{eq:A.16a}, we have
$$
R_1(s_0) \geq \liminf_{j \to \infty} R_1(s_{1,j})\geq  \liminf_{j \to \infty} R_1(s_{2,j}) +\delta_0 > \liminf_{s' \to s_0+} R_1(s').
$$
Combining with \eqref{eq:AA8}, we obtain
$$
\lim_{s' \to s_0-}(R_1 + R_2)(s') < (R_1+R_2)(s_0) \quad \text{ and }\quad (R_1+R_2)(s_0) > \liminf_{s' \to s_0+} (R_1 + R_2)(s').
$$
Since $R_1 + R_2$ is locally monotone at $s_0$, this is impossible. We have proved \eqref{eq:AA9}.

Having proved \eqref{eq:AA8} and \eqref{eq:AA9}, we may again take $(h_0,k_0)  =\frac{(-x_0,t_0)}{\|(-x_0,t_0)\|}$ and derive \eqref{eq:AA7a} and \eqref{eq:AA7}, so that the hypothesis (A2) can again be verified. 

Finally, if case (iii) and $R_2(s_0+) < R_2(s_0-)$ hold, then one can argue similarly that hypothesis (A2) holds with the choice of $(h_0,k_0) = \frac{(x_0,-t_0)}{\|(x_0,-t_0)\|}$. We omit the details. In conclusion, we have verified that $\underline w$ and $\overline w$ are respectively sub- and super-solutions of \eqref{eq:AA00} in $(0,\infty) \times (0,\infty)$, and hypotheses (A1)-(A3) hold. Hence, we can apply Theorem \ref{CP} to obtain
$\underline{w}(t,x)\leq \overline{w}(t,x)$ in $[0,\infty)\times[0,\infty)$, i.e.  $\underline\rho(s) \leq \overline\rho(s)$ in $[0,\infty)$.
\end{proof}

\section{Estimation for Proposition \ref{prop:2.1}}
In this section, we establish the upper estimate of $w^*(t,x)$ and lower estimate of $w_*(t,x)$. 

\begin{lemma}\label{lem:a4}
Assume that $\phi$ satisfies $({IC}_\mu)$ for some $\mu\in(0,\infty)$.	Let $w_*(t,x)$ and $w^*(t,x)$ be given by \eqref{eq:wstar}, then for any $\delta>0$,  there exist positive numbers $Q_1$ and $Q_2$ such that
	$$
	\max\{(\mu-\delta) x - Q_1 t,0\} \leq w_*(t,x)\le w^*(t,x)\le (\mu+\delta) x + Q_2 t\quad \text{ in }[0,\infty)\times [0,\infty).
	$$
	In particular, 
	$w_*(0,x)=w^*(0,x)=\mu x$ for each $x>0$.
	\end{lemma}
\begin{proof}
Suppose $\|\phi\|_{\infty}\le L_0$, by (H1), we see that $u^\ep$ is a supersolution of 
$\ep\partial_t u - \ep^2\partial_{xx}u+C u \geq 0 \text{ in }(0,\infty)\times\mathbb{R},
$ for some $C>0$. By (H5), there exists $\delta_0>0$, such that $u(t,0)\ge\delta_0>0$ for any $t\ge t_0$. 
Note that
$w^\ep(t,x)=-\ep\log u_\ep(t,x)$ satisfies
\begin{equation}\label{eq:wepsup}
 \begin{cases}
\partial_t w^\ep - \ep \partial_{xx} w^\ep + |\partial_x w^\ep|^2-C  \leq  0 &\text{ in }(t_0,\infty)\times (0,\infty),\\
 w^\ep(t,0) \leq \ep |\log \delta_0| &\text{ in }[t_0,\infty),\\
 w^\ep(t_0,x) <+\infty &\text{ in }[0,\infty).
 \end{cases}
\end{equation}
In view of $(IC)_\mu$, we know for any small $\delta\in(0,\mu)$, there exists $0<C_1<C_2$, such that 
$$C_1 e^{-(\mu+\delta)x}\le \phi(\theta, x)\le C_2e^{-(\mu-\delta)x} \text{ for } (\theta, x)\in[-\tau_0,0]\times[0,\infty).$$
Therefore, $$(\mu-\delta)x-\ep\log C_2\leq w^\ep(\theta,x)\leq (\mu+\delta)x-\ep\log C_1, \text{ for } (\theta, x)\in[-\ep\tau_0,0]\times[0,\infty).$$
Next, we define $\bar z^\ep=(\mu+\delta)x+Q_1(t+\ep),$ where $Q_1$ is chosen to be
$$Q_1=\max\left\{\sup\limits_{t\in[-\tau_0,\infty)}[-\log u(t,0)], |\log C_1|, C\right\},$$ 
which is finite in view of \rm{(H5)}. 
Then we have 
$$w^\ep(t,0)\le \bar z^\ep(t,0) \text{ for all } t\ge-\ep\tau_0, \,\text{ and }\, w^\ep(\theta,x)\le z^\ep(\theta,x), \text{ for all } (t,x)\in[-\ep\tau_0,0]\times \R_+.$$
By comparison principle ($\bar{z}^\ep$ is super-solution of the first equation in \eqref{eq:wepsup}), 
\begin{equation}\label{eq:zsup}
   w^\ep(t,x)\le \overline{z}^\ep(t,x) = (\mu+\delta)x+Q_1(t+\ep), \text{ for } (t,x)\in[-\ep\tau_0,0]\times\R_+. 
\end{equation}

Next, let $p \mapsto \lambda(p)$ be 
given by the implicit formula\begin{equation}
\Delta(\lambda,p):=-\lambda+p^2+|\partial_vf_2(\cdot,\cdot,0)|_{\infty}\int_{0}^{\tau_0}\int_{\R }\Gamma(\tau,y) e^{p y-\lambda \tau}  dy d\tau=0.
\end{equation}
Then one can similarly define
$\underline{z}^\ep(t,x)=(\mu-\delta)x-Q_2(t+\ep)$
with $Q_2=\max\{|\log C_2|, \lambda(\mu-\delta)\}$. 
By comparison, we have
\begin{equation}\label{eq:zsub}
    \max\{\underline{z}^\ep(t,x), -\ep\log (\|\phi\|_\infty+L_0)\}\le w^\ep(t,x) \text{ in } [-\ep\tau_0,\infty)\times\R_+.
\end{equation}
Combining \eqref{eq:zsup} and \eqref{eq:zsub}, and letting $\ep \to 0$, we have 
$$\max\{(\mu-\delta)x-Q_2t, 0\}\le w_*(t,x)\le w^*(t,x)\le (\mu+\delta)x+Q_1t  \text{ in } \R^2_+.$$	
Setting $t \to 0$ and letting $\delta \to 0$, it follows that
$\mu x\le w_*(0,x)\le w^*(0,x)\le \mu x$ for all $x>0$. 
\end{proof}	
\begin{remark}\label{rmk:largesp}
By \eqref{eq:zsub}, there exists $\overline{s} > 2 Q_2/(\mu - \delta)$ such that 
$$
w^\epsilon(t,x) \geq \max\{ Q_2 t, - O(\epsilon)\} \quad \text{ when }x \geq \overline{s} t.
$$
This implies $\displaystyle \lim\limits_{t \to \infty} \sup_{x \geq \overline s t} u(t,x) = 0.$
\end{remark}

\begin{lemma}\label{lem:a5} Assume that $\phi$ satisfies $({IC}_\infty)$.	Let $w_*(t,x)$  be given by \eqref{eq:wstar}, then 
\begin{equation}\label{eq:BB0}
w^*(0,x) = w_*(0,x)=\infty \quad \text{ for each }x >0.
\end{equation}
	\end{lemma}
\begin{proof}
Given a solution $u(t,x)$ of \eqref{Model} with compactly supported initial data $\phi$. Fix $\mu>0$, there exists $C_2>0$ such that 
$\phi(\theta,x)\le C_2e^{-\mu x}$ for any $(\theta,x)\in[-\tau_0,0]\times\R$. By repeating the proof of Lemma \ref{lem:a4}, we obtain a constant $Q_2=Q_2(\mu)>0$ such that 
$$\max\{\mu x-Q_2t, 0\}\le w_*(t,x)\, \text{ in } \R^2_+.$$
Fix $x>0$, we take $t \to 0$ to get
$$\mu x\le w_*(0,x).$$
Since the above holds for each $\mu >0$, we can take $\mu \to\infty$ to deduce $w_*(0,x) = \infty$ for each $x>0$. Notice that $w^*(0,x) \geq w_*(0,x)$ by construction, we obtain \eqref{eq:BB0}.
\end{proof}

\end{document}